\documentclass[reqno, 12pt]{article}

\pdfoutput=1

\usepackage{enumerate}
\usepackage{latexsym}
\usepackage[centertags]{amsmath}
\usepackage{amsfonts}
\usepackage{amssymb}
\usepackage{amsthm}
\usepackage{newlfont}
\usepackage{graphics}
\usepackage{color}
\usepackage{float}
\usepackage{diagbox}
\usepackage{tikz}
\usepackage{xcolor}
\usepackage{tocloft}
\usepackage{titlesec}
\usepackage{graphicx}
\usepackage{subfig}

\usepackage{booktabs}
\usepackage{algorithm}
\usepackage{algorithmicx}
\usepackage{algpseudocode}
\usepackage{url}

\usepackage[pagebackref]{hyperref}
\usetikzlibrary{arrows.meta, calc}

\usepackage{hyperref}
\usepackage{longtable}
\usepackage{rotating}
\usepackage{multirow}
\usepackage{extarrows}
\usepackage[sort,compress,numbers]{natbib}
\usepackage[utf8]{inputenc}

\newtheorem{thm}{Theorem}[section]
\newtheorem{cor}[thm]{Corollary}
\newtheorem{lem}[thm]{Lemma}
\newtheorem{prop}[thm]{Proposition}

\newtheorem{cnj}[thm]{Conjecture}
\theoremstyle{mydefinition}
\newtheorem{dfn}[thm]{Definition}
\theoremstyle{myremark}
\newtheorem{rem}[thm]{Remark}

\newtheorem{prob}[thm]{Problem}

\newtheorem{Fact}{Fact}
\allowdisplaybreaks[4]

\newcommand\CT{\mathop{\mathrm{CT}}}

\makeatletter
\let\c@algorithm\c@thm
\makeatother
\numberwithin{algorithm}{section}

\title{Proof of a Conjecture on Young Tableaux with Walls}

\author{Zhicong Lin$^{\color{blue} \dag}$, Feihu Liu$^{\color{blue} \ddag}$\thanks{Corresponding author}, Jiahang Liu$^{\color{blue} \S}$, Jing Liu$^{\color{blue}\P}$, and Guoce Xin$^{\color{blue} \maltese}$
\\[2mm]
{\small $^{\color{blue} \dag, \P}$ Research Center for Mathematics and Interdisciplinary Sciences,}\\[-0.8ex]
{\small Shandong University, Qingdao 266237, P.R.~China}\\
{\small $^{\color{blue} \ddag}$ Center for Combinatorics, LPMC,}\\[-0.8ex]
{\small Nankai University, Tianjin 300071, P.R.~China}\\
{\small $^{\color{blue} \S,\maltese}$ School of Mathematical Sciences,}\\[-0.8ex]
{\small Capital Normal University, Beijing 100048, P.R.~China}\\
{\small $^{\color{blue}  \P}$ Department of Mathematics,}\\[-0.8ex]
{\small Sungkyunkwan University, Suwon  16419, South Korea}\\
{\small {\color{blue} $^\dag$} Email address: linz@sdu.edu.cn}\\
{\small {\color{blue} $^\ddag$} Email address: liufeihu7476@163.com}\\
{\small {\color{blue} $^\S$} Email address: jiahang-liu@foxmail.com}\\
{\small {\color{blue} $^\P$} Email address: lsweet@mail.sdu.edu.cn}\\
{\small {\color{blue} $^\maltese$} Email address: guoce\_xin@163.com}
}

\date{\today}

\begin{document}

\maketitle

\begin{abstract}
Banderier, Marchal, and Wallner considered Young tableaux with walls, which are similar to standard Young tableaux, except that local decreases are allowed at some walls. In this work, we prove a conjecture of Fuchs and Yu concerning the enumeration of two classes of three-row Young tableaux with walls. Together with the work of Chang, Fuchs, Liu, Wallner, and Yu, our result verifies a conjecture of Pons and Batle on tree-child networks. This conjecture had been regarded as a specific and challenging problem in the phylogenetics community until its resolution in the present work.
\end{abstract}

\noindent
\begin{small}
\emph{2020 Mathematics subject classification}: Primary 05A15; Secondary 05A10; 05A17; 06A07.
\end{small}

\noindent
\begin{small}
\emph{Keywords}: Young tableaux; Tree-child networks; Exact enumeration; Recursions; Generating functions; Linear extensions.
\end{small}

\tableofcontents

\section{Introduction}
Young tableaux with walls were introduced by
Banderier, Marchal, and  Wallner~\cite{C.Banderier18}, and their enumeration is quite challenging.
The main objective of this paper is to prove an enumerative conjecture (see Theorem~\ref{MainConject}) of Fuchs and Yu~\cite{Fang}
concerning  two classes of three-row Young tableaux with walls. Tree-child networks are one of the most prominent network classes for modeling evolutionary processes involving reticulation events, see \cite{PB,CZ,CFY}. Our result, together with the work by Chang, Fuchs, Liu, Wallner, and Yu~\cite{CFLWY}, leads to the verification of a conjecture due to Pons and Batle~\cite[Conjecture~1]{PB} on counting  tree-child networks. Before proceeding to state our main result, we need to introduce some notation and definitions related to Young tableaux with walls.

Let $\lambda=(\lambda_1,\lambda_2,\ldots)$ be a partition of $m$, denoted $\lambda \vdash m$.
The \emph{Young diagram} of $\lambda$ is obtained by drawing a collection of left-justified arrays of $m$ cells, where row $i$ has $\lambda_i$ cells reading from bottom to top.
For the sake of convenience, the Young diagram of $\lambda$ is also denoted by $\lambda$.
A \emph{standard Young tableau} of $\lambda$ is an assignment of natural numbers $1,2,\ldots,m$ to the cells of the Young diagram, each number appearing once, such that rows are increasing from left to right, and columns are increasing from bottom to top.

Recently, Banderier, Marchal, and Wallner~\cite{C.Banderier18,C.Banderier21} introduced a variation of standard Young tableaux, called \emph{Young tableaux with walls}. In this setting, a bold red edge, called a \emph{wall}, is placed horizontally or vertically between two adjacent cells where a decrease is allowed (but not required); see Fig.~\ref{tab:young}~(a) for an example. A \emph{Young tableau with walls} of shape $\lambda$ is a filling of the cells with $1,2,\dots,|\lambda|$, each used once, such that entries increase along each row and column as in standard Young tableaux, except that entries separated by a wall need not be increasing; see Fig.~\ref{tab:young}~(b).

\begin{figure}[htbp]
\centering
\subfloat[]{
\begin{tikzpicture}[
    cell/.style={draw, minimum size=1cm, anchor=center},
    wall/.style={line width=2pt, red}]
\node[cell] (a1) at (0,0) {};
\node[cell] (a2) at (1,0) {};
\node[cell] (a3) at (2,0) {};
\node[cell] (a4) at (3,0) {};

\node[cell] (b1) at (0,-1) {};
\node[cell] (b2) at (1,-1) {};
\node[cell] (b3) at (2,-1) {};
\node[cell] (b4) at (3,-1) {};
\node[cell] (b5) at (4,-1) {};

\node[cell] (c1) at (0,-2) {};
\node[cell] (c2) at (1,-2) {};
\node[cell] (c3) at (2,-2) {};
\node[cell] (c4) at (3,-2) {};
\node[cell] (c5) at (4,-2) {};
\node[cell] (c6) at (5,-2) {};
\node[cell] (c7) at (6,-2) {};

\draw[wall] (0.5,0.5) -- (0.5,-0.5);
\draw[wall] (1.5,0.5) -- (1.5,-0.5);

\draw[wall] (1.5,-1.5) -- (2.5,-1.5);
\draw[wall] (3.5,-1.5) -- (3.5,-2.5);
\end{tikzpicture}}
\hspace{15pt}
\subfloat[]{
\begin{tikzpicture}[
    cell/.style={draw, minimum size=1cm, anchor=center},
    wall/.style={line width=2pt, red}]
\node[cell] (a1) at (0,0) {10};
\node[cell] (a2) at (1,0) {6};
\node[cell] (a3) at (2,0) {8};
\node[cell] (a4) at (3,0) {14};

\node[cell] (b1) at (0,-1) {3};
\node[cell] (b2) at (1,-1) {4};
\node[cell] (b3) at (2,-1) {5};
\node[cell] (b4) at (3,-1) {13};
\node[cell] (b5) at (4,-1) {15};

\node[cell] (c1) at (0,-2) {1};
\node[cell] (c2) at (1,-2) {2};
\node[cell] (c3) at (2,-2) {7};
\node[cell] (c4) at (3,-2) {12};
\node[cell] (c5) at (4,-2) {9};
\node[cell] (c6) at (5,-2) {11};
\node[cell] (c7) at (6,-2) {16};

\draw[wall] (0.5,0.5) -- (0.5,-0.5);
\draw[wall] (1.5,0.5) -- (1.5,-0.5);

\draw[wall] (1.5,-1.5) -- (2.5,-1.5);
\draw[wall] (3.5,-1.5) -- (3.5,-2.5);
\end{tikzpicture}}
\caption{A Young diagram with walls and a Young tableau with walls.}
\label{tab:young}
\end{figure}

The enumeration of Young tableaux with walls is a difficult problem. Via bijections, hook-length-like formulas, and the density method, some results for the enumeration of rectangular Young tableaux with walls were proved by Banderier, Marchal, and Wallner~\cite{C.Banderier18}. Banderier and Wallner~\cite{C.Banderier21} further  considered Young tableaux with periodic walls and reported more enumerative results and  bijections with trees, lattice paths, or permutations.
By employing the kernel method and the lattice path enumeration theory, Liu and Xin~\cite{LiuXin-YTW} obtained the generating function for a class of two-row Young tableaux with periodic walls.

Here we are concerned  with two specified  classes of three-row Young tableaux with walls. Let $(n,n,k)\vdash2n+k$ with $n\geq k\geq 0$.
We place a vertical wall (bold red edge) between each pair of adjacent cells at the bottom of the Young diagram of $(n,n,k)$.
An example is shown in Fig.~\ref{tab:ank} for $n=7$ and $k=4$. Let $a_{n,k}$ denote the number of all such Young tableaux with walls.
\begin{figure}[htbp]
\centering
\begin{tikzpicture}[
    cell/.style={draw, minimum size=1cm, anchor=center},
    wall/.style={line width=2pt, red}]
\node[cell] (a1) at (0,0) {};
\node[cell] (a2) at (1,0) {};
\node[cell] (a3) at (2,0) {};
\node[cell] (a4) at (3,0) {};

\node[cell] (b1) at (0,-1) {};
\node[cell] (b2) at (1,-1) {};
\node[cell] (b3) at (2,-1) {};
\node[cell] (b4) at (3,-1) {};
\node[cell] (b5) at (4,-1) {};
\node[cell] (b6) at (5,-1) {};
\node[cell] (b7) at (6,-1) {};

\node[cell] (c1) at (0,-2) {};
\node[cell] (c2) at (1,-2) {};
\node[cell] (c3) at (2,-2) {};
\node[cell] (c4) at (3,-2) {};
\node[cell] (c5) at (4,-2) {};
\node[cell] (c6) at (5,-2) {};
\node[cell] (c7) at (6,-2) {};

\draw[wall] (0.5,-1.5) -- (0.5,-2.5);
\draw[wall] (1.5,-1.5) -- (1.5,-2.5);
\draw[wall] (2.5,-1.5) -- (2.5,-2.5);
\draw[wall] (3.5,-1.5) -- (3.5,-2.5);
\draw[wall] (4.5,-1.5) -- (4.5,-2.5);
\draw[wall] (5.5,-1.5) -- (5.5,-2.5);
\end{tikzpicture}
\caption{A Young diagram of shape $(7,7,4)$ with walls at the bottom.}
\label{tab:ank}
\end{figure}
In particular, a closed-form expression for the sequence $(a_{n,n})_{n\geq 0}$ that was registered  as~\cite[A213863]{Sloane23} in the OEIS remains unknown.
Fuchs, Yu, and Zhang \cite{FYZ} derived an asymptotic expression for $a_{n,n}$ by using the methods introduced in \cite{EFW21}.
See also \cite[Theorem 4.1]{C.Banderier21}.
Pons and Batle \cite[Proposition 7]{PB} established the asymptotic behavior of $a_{n,k}$ as $n\rightarrow \infty$ for fixed $k$.
On the other hand, it is easy to see that
\begin{equation}\label{eq:an0}
a_{n,0}=(2n-1)!!,
\end{equation}
the odd double factorial, while the sequence $(a_{n,1})_{n\geq 1}$ corresponds to A122649 in the OEIS~\cite{Sloane23}.
We adopt the convention $(-1)!!:=1$ throughout the paper.

Moreover, by considering the position of the largest label in such a Young tableau with walls, one
can derive a recurrence relation for $a_{n,k}$:
\begin{equation}\label{rec:ank}
a_{n,k}=a_{n,k-1}+(2n+k-1)\cdot a_{n-1,k},\quad n\geq k\geq 1,
\end{equation}
with the initial condition $a_{n,0}=(2n-1)!!$ for $n\ge 0$, and the boundary condition $a_{n-1,n}=0$ for $n\ge 1$.
The first few values of $a_{n,k}$ are provided in Table~\ref{tabb-Ank}, for the sake of convenience.
\begin{tiny}
\begin{table}
\centering
\caption{Values of $a_{n,k}$ for $0\le k\le n\le 6$; see Fig.~\ref{tab:ank} for the general shape (shown there with $n=7,k=4$).}

\label{tabb-Ank}
    	\begin{tabular}{c||c|c|c|c|c|c|c}
    		\hline \hline
$n\setminus k$ & 0 & 1 & 2 & 3 & 4 & 5 & 6 \\
    		\hline
$0$ & 1 &   &   &   &   &  &  \\
    		\hline
$1$ & 1 &  1 &   &   &   &  &  \\
    		\hline
$2$ & 3 &  7 &  7 &   &   &  &  \\
    		\hline
$3$ & 15 & 57  & 106  & 106  &   &  &  \\
    		\hline
$4$ & 105 & 561  & 1515  & 2575  & 2575  &  &  \\
    		\hline
$5$ & 945 &  6555 & 23220  &  54120 & 87595  &  87595 &  \\
    		\hline
$6$ & 10395 & 89055  &  390915  &  1148595  &  2462520  & 3864040 & 3864040 \\
    		\hline
    	\end{tabular}
\end{table}
\end{tiny}

Next, we introduce the second kind of Young tableaux with walls. Let $(n,n,n)\vdash 3n$ and $n\geq k$. We place a vertical wall (bold red edge) between each pair of adjacent cells at the bottom row of the Young diagram of $(n,n,n)$.
By removing any $n-k$ bottom cells, we obtain a family of ${n\choose k}$ \emph{deformed Young diagrams with walls}.
One such diagram is illustrated in Fig.~\ref{tab:bnk} for $n=7$ and $k=4$.
\begin{figure}[htbp]
\centering
\begin{tikzpicture}[
    cell/.style={draw, minimum size=1cm, anchor=center},
    wall/.style={line width=2pt, red}]
\node[cell] (a1) at (0,0) {};
\node[cell] (a2) at (1,0) {};
\node[cell] (a3) at (2,0) {};
\node[cell] (a4) at (3,0) {};
\node[cell] (a5) at (4,0) {};
\node[cell] (a6) at (5,0) {};
\node[cell] (a7) at (6,0) {};

\node[cell] (b1) at (0,-1) {};
\node[cell] (b2) at (1,-1) {};
\node[cell] (b3) at (2,-1) {};
\node[cell] (b4) at (3,-1) {};
\node[cell] (b5) at (4,-1) {};
\node[cell] (b6) at (5,-1) {};
\node[cell] (b7) at (6,-1) {};

\node[cell] (c1) at (1,-2) {};
\node[cell] (c2) at (3,-2) {};
\node[cell] (c3) at (4,-2) {};
\node[cell] (c4) at (5,-2) {};

\draw[wall] (3.5,-1.5) -- (3.5,-2.5);
\draw[wall] (4.5,-1.5) -- (4.5,-2.5);
\end{tikzpicture}
\caption{A case in a family of deformed Young diagrams with walls for $n=7$ and $k=4$.}
\label{tab:bnk}
\end{figure}
Similarly, a \emph{deformed Young tableau with walls} of this type is a filling of the cells by labels $1,2,\dots, 2n+k$, such that each label appears exactly once and the labels are increasing left-to-right and bottom-to-top, respectively, but two labels separated by a wall need not be increasing.
Let $b_{n,k}$ denote the total number of all deformed Young tableaux with walls on this family of deformed Young diagrams with walls.
It is well known that
\begin{equation}\label{eq:bn0}
b_{n,0}=\frac{1}{n+1}\binom{2n}{n}=:\mathrm{Cat}(n),\qquad n\geq 0,
\end{equation}
where $\mathrm{Cat}(n)$ is the famous {\em Catalan number} appearing as~\cite[A000108]{Sloane23} in the OEIS. The sequence $(b_{n,1})_{n\geq 1}$ was registered as A000531 in the OEIS~\cite{Sloane23}. Compared  to $a_{n,k}$, there exists no simple recursion for $b_{n,k}$, and computing $b_{n,k}$ appears to be more difficult than computing  $a_{n,k}$. A somewhat complicated  recurrence relation for a refinement of $b_{n,k}$ was proved in~\cite[Proposition~3.7]{CFLWY}. For the reader's convenience, we list the first few values of $b_{n,k}$ in Table~\ref{tabb-Bnk}.

\begin{tiny}
\begin{table}
    	\centering
\caption{Values of $b_{n,k}$ for $0\le k\le n\le 6$; see Fig.~\ref{tab:bnk} for the general shape (shown there with $n=7,k=4$).}
    \label{tabb-Bnk}
    	\begin{tabular}{c||c|c|c|c|c|c|c }
    		\hline \hline
$n\setminus k$ & 0 & 1 & 2 & 3 & 4 & 5 & 6 \\
    		\hline
$0$ & 1 &   &   &   &   &  &  \\
    		\hline
$1$ & 1 &  1 &   &   &   &  &  \\
    		\hline
$2$ & 2 &  7 &  7 &   &   &  & \\
    		\hline
$3$ & 5 & 38  & 106  & 106  &   &  & \\
    		\hline
$4$ & 14 & 187  & 1010  & 2575  & 2575  &  &  \\
    		\hline
$5$ & 42 &  874 & 7740  &  36080 & 87595  &  87595 &  \\
    		\hline
$6$ & 132  & 3958  &  52122  &  382865 &  1641680  & 3864040 & 3864040 \\
    		\hline
    	\end{tabular}
\end{table}
\end{tiny}

The main result of this paper is the following unexpected relationship between $a_{n,k}$ and $b_{n,k}$, which was originally conjectured by Fuchs and Yu~\cite{Fang}.
\begin{thm}[Fuchs and Yu's enumerative conjecture]\label{MainConject}
For $0\leq k\leq n$, we have
\begin{align}\label{JingLiu}
2^{n-k}\cdot a_{n,k}=(n-k+1)!\cdot b_{n,k}.
\end{align}
\end{thm}

\begin{cor}\label{Canject-bnk}
The $b_{n,k}$ satisfies the following recurrence relation:
$$b_{n,k}=\frac{n-k+2}{2}\cdot b_{n,k-1}+\frac{2(2n+k-1)}{n-k+1}\cdot b_{n-1,k}, \quad n\geq k\geq 1,$$
with the initial condition $b_{n,0}$ given by \eqref{eq:bn0}, and the boundary condition $b_{n-1,n}=0$ for $n\geq 1$.
\end{cor}
\begin{proof}
This follows from Theorem~\ref{MainConject} and the recurrence~\eqref{rec:ank}.
\end{proof}

It is straightforward to verify that
\begin{equation}\label{rela:oddcat}
2^n\cdot(2n-1)!!=(n+1)!\cdot\mathrm{Cat}(n),
\end{equation}
 from which we see that Theorem~\ref{MainConject} holds for $k=0$ in view of~\eqref{eq:an0} and~\eqref{eq:bn0}.
Thus, Theorem~\ref{MainConject} can be considered as an extension of the relationship~\eqref{rela:oddcat} between the odd factorial numbers and the Catalan numbers.
Theorem~\ref{MainConject} was inspired by recent work of Chang et al. in~\cite{CFLWY} and closely connected to a  conjecture of Pons and Batle~\cite[Conjecture~1]{PB} on counting  tree-child networks that we next elaborate on.

The class of tree-child networks is an increasingly prominent class of phylogenetic networks.
Following~\cite{PB}, a {\em phylogenetic network with $n$ leaves} is a rooted acyclic digraph without parallel edges such that all nodes belong to one of the following categories:
\begin{itemize}
\item[(1)] A unique {\em root} which has indegree 0 and outdegree 1.
\item[(2)] {\em Leaf nodes} which have indegree 1 and outdegree 0. They are bijectively labeled with elements in  $\{1,2,\ldots,n\}$.
\item[(3)] {\em Tree nodes} which have indegree 1 and outdegree 2.
\item[(4)] {\em Reticulation nodes} which have indegree 2 and outdegree 1.
\end{itemize}
\begin{figure}[htbp]
\centering
\begin{tikzpicture}[scale=1.3,
    node base/.style={circle, draw, fill=black, inner sep=0pt, minimum size=4pt}, 
    root node/.style={circle, draw, fill=white, inner sep=0pt, minimum size=4pt}, 
    child node/.style={draw, fill=red, inner sep=0pt, minimum size=4pt},
]

\node[root node] (root) at (0, 4) {}; 
\node[node base] (n1) at (0, 3.5) {}; 
\node[node base] (n2) at (-1, 2.5) {}; 
\node[node base] (n3) at (0.5, 3) {};
\node[child node] (n4) at (0, 2.5) {};
\node[node base] (n5) at (-0.5, 2) {};\node at (-0.5,1.3) {1};
\node[node base] (n6) at (-1.5, 2) {};\draw[->] (n2) -- (n6);\node at (-1.5,1.8) {3};
\node[node base] (n7) at (-0.5, 1.5) {};\draw[->] (n5) -- (n7);

\node[node base] (n8) at (1, 1.5) {};\draw[->] (n4) -- (n8);\node at (1,0.8) {4};
\node[node base] (n9) at (1.5, 2) {};\draw[->] (n3) -- (n9);\draw[->] (n9) -- (n8);
\node[node base] (n10) at (1, 1) {};\draw[->] (n8) -- (n10);
\node[node base] (n11) at (2, 1.5) {};\draw[->] (n9) -- (n11);\node at (2,1.3) {2};

\draw[->] (root) -- (n1);

\draw[->] (n1) -- (n2);\draw[->] (n1) -- (n3);
\draw[->] (n3) -- (n4);\draw[->] (n4) -- (n5);
\draw[->] (n2) -- (n5);

\node[root node] (r1) at (5, 4) {}; 
\node[node base] (v1) at (5, 3.5) {}; 
\node[node base] (v2) at (4, 2.5) {}; 
\node[node base] (v3) at (5.5, 3) {};
\node[node base] (v4) at (5, 2.5) {};\node[node base] (r2) at (5.5, 2) {};
\node[node base] (r3) at (5.2, 1.5) {};\draw[->] (r2) -- (r3);\node at (5.2,1.3) {5};
\node[node base] (v5) at (4.5, 2) {};\node at (4.5,1.3) {1};\draw[->] (v4) -- (r2);

\node[node base] (v6) at (3.5, 2) {};\draw[->] (v2) -- (v6);\node at (3.5,1.8) {3};
\node[node base] (v7) at (4.5, 1.5) {};\draw[->] (v5) -- (v7);

\node[node base] (v8) at (6, 1.5) {};\draw[->] (v4) -- (v8);\node at (6,0.8) {4};
\node[node base] (v9) at (6.5, 2) {};\draw[->] (v3) -- (v9);\draw[->] (v9) -- (v8);
\node[node base] (v10) at (6, 1) {};\draw[->] (v8) -- (v10);
\node[node base] (v11) at (7, 1.5) {};\draw[->] (v9) -- (v11);\node at (7,1.3) {2};

\draw[->] (r1) -- (v1);

\draw[->] (v1) -- (v2);\draw[->] (v1) -- (v3);
\draw[->] (v3) -- (v4);\draw[->] (v4) -- (v5);
\draw[->] (v2) -- (v5);
\end{tikzpicture}

\caption{
Two phylogenetic networks: only the right one is tree-child.}
\label{fig:phylogenetic_network}
\end{figure}
A phylogenetic network is {\em tree-child} if, for each non-leaf node, at least one of its
children is a tree node or leaf.
See Fig.~\ref{fig:phylogenetic_network} for two examples of phylogenetic networks, where only the right one is tree-child. In a phylogenetic network with $n$ leaves, $k$ reticulation nodes, and $t$ tree nodes, we have
$$n+k=t+1$$
since the sum of the outdegrees equals that of the indegrees.
Phylogenetic networks and tree-child networks have been extensively investigated from the enumerative aspect~\cite{Bouvel,CFY,CFLWY,CZ,FLY,FYZ,FYZ0,PB}. In particular, Pons and Batle~\cite{PB} proposed the following enumerative conjecture.

\begin{cnj}[\text{Pons and Batle~\cite[Conjecture~1]{PB}}] \label{Pons-Batle}
\label{conj:PB}
Let $\mathcal{TC}_{n,k}$ be the set of all  tree-child networks with $n$ leaves and $k$ reticulation nodes. Then
\begin{equation}\label{eq:conPB}
|\mathcal{TC}_{n,k}|=\frac{n!}{(n-k)!}\cdot a_{n-1,k}.
\end{equation}
\end{cnj}

In fact, $a_{n,k}$ is interpreted in~\cite{PB} as some Yamanouchi-like  words, which are in one-to-one correspondence (see~\cite[Proposition~7.10.3~(d)]{RP.Stanley2024}) with the  Young tableaux with walls of shape $(n,n,k)$. Conjecture~\ref{conj:PB} stems from the work by Fuchs, Yu, and Zhang~\cite{FYZ} and has received considerable attention in recent years~\cite{CFLWY, FLY}. In particular, Fuchs, Liu, and Yu~\cite{FLY} confirmed Conjecture~\ref{conj:PB}  for one-component tree-child
networks which are an important building block of (general) tree-child networks (see~\cite{CZ}). Note that  several interesting consequences of Conjecture~\ref{conj:PB} have been discussed in~\cite{PB,CZ}; see also Section~\ref{apps:thm1.1} for some of them.

Conjecture~\ref{conj:PB} now follows as a corollary of our main result,  Theorem~\ref{MainConject}.
\begin{proof}[{\bf Proof of Conjecture~\ref{conj:PB}}]
It was proved in~\cite[Theorem~3.4]{CFLWY} that
$$
|\mathcal{TC}_{n,k}|=\frac{n!}{2^{n-k-1}}\cdot b_{n-1,k},
$$
which,  when combined with  Theorem~\ref{MainConject},  implies~\eqref{eq:conPB}.
\end{proof}

The rest of this paper is organized as follows.
Section~\ref{proof:conj1} is devoted to the proof of Theorem~\ref{MainConject}.
Section~\ref{Section-Generating-function} focuses on investigating the generating function for $b_{n,k}$.  Via linear extensions of some posets and the principle of inclusion-exclusion, we derive in Section~\ref{sec:LExten} a recurrence for $b_{n,k}$.
Finally, we end this paper with concluding remarks and related open problems in Section~\ref{Section-concluding-remark}.
Many of the \texttt{Maple} computation procedures involved in this paper can be found on the second author's GitHub page \cite{MapleTC}.

\section{Proof of Theorem~\ref{MainConject}}
\label{proof:conj1}

This section is devoted to the proof of Theorem~\ref{MainConject}. Since our proof is somewhat intricate, we first intend to outline the key steps of our proof:
\begin{enumerate}
\item Starting from~\eqref{rec:ank}, we derive a semi-closed-form expression for $a_{n,k}$ (see Theorem~\ref{Semi-closed-formula-Ank}).
\item We introduce an extension  of $b_{n,k}$, denoted $b_{n,m,k}$,  by generalizing  the deformed Young tableaux with walls. The extension  $b_{n,m,k}$ admits a nice recurrence relation (see~Proposition~\ref{Definition-bnk-Proposition}):
\begin{equation}\label{rec:bnmk2}
b_{n,m,k}=(m-k+1)b_{n,m,k-1} + b_{n,m-1,k} + b_{n-1,m,k},
\end{equation}
which can be applied to compute $b_{n,k}$ efficiently, as  $b_{n,k}=b_{n,n,k}$.
\item Inspired  by the favorable recursion for $b_{n,m,k}$ and the conjectured relationship~\eqref{JingLiu} with $a_{n,k}$, we introduce a variant of $b_{n,m,k}$, denoted $\omega_{n,m,k}$, which satisfies a variant of~\eqref{rec:bnmk2} with different initial conditions.
\item To finish the proof, in view of the construction of $\omega_{n,m,k}$ and Theorem~\ref{Semi-closed-formula-Ank}, it remains to show that $$\omega_{n,m,k}=b_{n+m,m,k}.$$
    To this end, we establish two key lemmas to verify that $\omega_{n,m,k}$ and $b_{n+m,m,k}$ share the same initial conditions, a part that constitutes the most technical aspect of our proof.
\end{enumerate}

\subsection{A semi-closed-form expression for $a_{n,k}$ }
This subsection aims to prove the following semi-closed-form expression for $a_{n,k}$.

\begin{thm}\label{Semi-closed-formula-Ank}
Let $0\leq k\leq n$. We have
\begin{align*}
a_{n,k}=\sum_{i=0}^k \frac{\gamma_{k-i}}{i!}\cdot (2n+k+i-1)!!,
\end{align*}
where $\gamma_k$ are rational numbers satisfying the recurrence relation
\begin{align}\label{Recurren-Ck0}
\gamma_k=-\frac{1}{(3k-3)!!} \left(\sum_{i=1}^k \frac{\gamma_{k-i}}{i!}\cdot (3k+i-3)!! \right), \quad k\geq 1,
\end{align}
with the initial condition $\gamma_0=1$.
\end{thm}

\begin{rem}
Theorem~\ref{Semi-closed-formula-Ank} was previously obtained in~\cite[Proposition~6]{PB}; we discovered it independently.
Starting from~\eqref{rec:ank}, we derived the recurrence
\begin{equation}\label{eq:recank2}
a_{n,k}=a_{n,k-1}+\sum_{i=k}^{n-1}\prod_{j=i}^{n-1}(2(j+1)+k-1)\cdot a_{i,k-1},
\end{equation}
which allowed us to compute explicit formulas for $a_{n,k}$ for small fixed $k$ (see Proposition~\ref{proposition-ank-init-term}) and ultimately led to the closed form in Theorem~\ref{Semi-closed-formula-Ank}.
\end{rem}

\medskip
\noindent
\textbf{The prime check heuristic.}
At several points in the paper, we will guess a formula from computational data and then prove it by induction.
The heuristic that guides us in this process is the following \emph{prime check}: we compute the relevant quantities for small parameter values and factor the resulting integers.
If the factors are consistently small primes, this signals that a simple closed form may exist and helps us recognize the underlying pattern; conversely, the appearance of large prime factors suggests that no simple formula is likely.

For instance, applying the prime check to the diagonal sequence $a_{n,n}$ did not reveal an obvious pattern, as the values exhibited large prime factors; this suggested that no simple closed form was available for $a_{n,n}$ alone.
This was indeed the case: only after we derived the full $a_{n,k}$ formula did we obtain, as a special case, the following expression for $a_{n,n}$ (see \cite{FYZ} or \cite{C.Banderier21} for an asymptotic formula):
\begin{equation*}
a_{n,n}=\sum_{i=0}^n \frac{\gamma_{n-i}}{i!}\cdot (3n+i-1)!!,
\end{equation*}
where $\gamma_j$ satisfies \eqref{Recurren-Ck0} with $\gamma_0=1$.
The fact that the prime check gave no hint of this formula made its eventual discovery all the more surprising.
Nonetheless, the recurrence~\eqref{eq:recank2} was sufficient to conjecture the full expression involving the $\gamma_k$ sequence in Theorem~\ref{Semi-closed-formula-Ank}.

Since Theorem~\ref{Semi-closed-formula-Ank} is crucial in proving Theorem~\ref {MainConject}, we include its proof below, for the sake of completeness.
The following method of proof, namely that two sequences satisfy the same recurrence relations and the same initial conditions, is our key method and will be used repeatedly, especially in the proof of our main result.

\begin{proof}[{\bf Proof of Theorem~\ref{Semi-closed-formula-Ank}}]
For convenience, we set
\begin{align*}
\tilde{a}_{n,k}=\sum_{i=0}^k \frac{\gamma_{k-i}}{i!}\cdot (2n+k+i-1)!!.
\end{align*}
We aim to show that $\tilde{a}_{n,k}$ and $a_{n,k}$ share the same recurrence relation and initial conditions.
Firstly, we have $\tilde{a}_{n,0}=(2n-1)!!=a_{n,0}$.
Secondly, the sequence $\tilde{a}_{n,k}$ satisfies the following equation:
\begin{align*}
(2n+k-1)\tilde{a}_{n-1,k}& =(2n+k-1)\sum_{i=0}^k \frac{\gamma_{k-i}}{i!}\cdot (2n+k+i-3)!!
\\&=(2n+k+i-1-i)\sum_{i=0}^k \frac{\gamma_{k-i}}{i!}\cdot (2n+k+i-3)!!
\\&=\sum_{i=0}^k \frac{\gamma_{k-i}}{i!}\cdot (2n+k+i-1)!!-\sum_{i=0}^k \frac{i\cdot\gamma_{k-i}}{i!}\cdot (2n+k+i-3)!!
\\&=\tilde{a}_{n,k}-\sum_{i=1}^{k} \frac{\gamma_{k-i}}{(i-1)!}\cdot (2n+k+i-3)!!
\\&=\tilde{a}_{n,k}-\sum_{i=0}^{k-1} \frac{\gamma_{k-i-1}}{i!}\cdot (2n+k+i-2)!!
\\&=\tilde{a}_{n,k}-\tilde{a}_{n,k-1}.
\end{align*}
Therefore, we obtain
$$\tilde{a}_{n,k}=\tilde{a}_{n,k-1}+(2n+k-1)\tilde a_{n-1,k}.$$
In view of~\eqref{rec:ank},  $\tilde{a}_{n,k}$ and $a_{n,k}$ share the same recurrence relation.
Finally, it follows from~\eqref{Recurren-Ck0} that
$$\tilde{a}_{k-1,k}=\sum_{i=0}^{k} \frac{\gamma_{k-i}}{i!} (3k+i-3)!!=0.$$
This is consistent with $a_{k-1,k}=0$, and thus $a_{n,k}=\tilde{a}_{n,k}$.
This completes the proof.
\end{proof}

\subsection{An extension of $b_{n,k}$ and the introduction of its variant $\omega_{n,m,k}$}
\label{Section-Extension-of-bnk}

In this section, we consider a natural  extension of $b_{n,k}$ and introduce  its variant $\omega_{n,m,k}$ based on Theorem~\ref{Semi-closed-formula-Ank}.

Let $n\geq m\geq k\geq 0$ and $(n,m,m)\vdash n+2m$.
We place a vertical wall (bold red edge) between each pair of adjacent cells at the top row of the Young diagram of $(n,m,m)$.
By removing any $m-k$ cells at the top, we obtain a family of deformed Young diagrams with walls.
One such diagram is illustrated in Fig.~\ref{tab:bnmk} for $n=9$, $m=7$, and $k=4$.
Similarly, a deformed Young tableau with walls of this type is a filling of the cells by labels $1,2,\dots, n+m+k$, such that each label appears exactly once and the labels are increasing left-to-right and bottom-to-top, respectively, but two labels separated by a wall need not be increasing.
Let $b_{n,m,k}$ denote the total number of deformed Young tableaux with walls on this family of deformed Young diagrams with walls.

\begin{figure}[htbp]
\centering
\begin{tikzpicture}[
    cell/.style={draw, minimum size=1cm, anchor=center},
    wall/.style={line width=2pt, red}]

\node[cell] (a2) at (1,0) {};
\node[cell] (a4) at (3,0) {};
\node[cell] (a5) at (4,0) {};
\node[cell] (a6) at (5,0) {};

\node[cell] (b1) at (0,-1) {};
\node[cell] (b2) at (1,-1) {};
\node[cell] (b3) at (2,-1) {};
\node[cell] (b4) at (3,-1) {};
\node[cell] (b5) at (4,-1) {};
\node[cell] (b6) at (5,-1) {};
\node[cell] (b7) at (6,-1) {};

\node[cell] (c1) at (0,-2) {};
\node[cell] (c2) at (1,-2) {};
\node[cell] (c3) at (2,-2) {};
\node[cell] (c4) at (3,-2) {};
\node[cell] (c5) at (4,-2) {};
\node[cell] (c6) at (5,-2) {};
\node[cell] (c7) at (6,-2) {};
\node[cell] (c8) at (7,-2) {};
\node[cell] (c9) at (8,-2) {};

\draw[wall] (3.5,0.5) -- (3.5,-0.5);
\draw[wall] (4.5,0.5) -- (4.5,-0.5);
\end{tikzpicture}
\caption{A deformed Young diagram with walls for the parameter values $n=9$, $m=7$, and $k=4$.}
\label{tab:bnmk}
\end{figure}

When $k=0$, the hook-length formula \cite[Corollary 7.21.6]{RP.Stanley2024} for counting standard Young tableaux gives
$$b_{n,m,0}=\frac{(n+m)!(n-m+1)}{m!(n+1)!}.$$
Let
\begin{align*}
C(x)=\sum_{n\geq 0}\mathrm{Cat}(n) x^n=\sum_{n\geq 0}\frac{1}{n+1}\binom{2n}{n} x^n=\frac{1-\sqrt{1-4x}}{2x}
\end{align*}
be the generating function of Catalan numbers.
Summing the above closed form over $0\le m\le n$ gives
$$\sum_{0\leq m\leq n}b_{n,m,0}x^ny^m=\frac{1-yC(xy)}{1-x-y}.$$
This identity is a classical application of the kernel method to the recurrence for $b_{n,m,0}$, which leads us to adapt the method in Subsection~\ref{subsection-recurrence-gf} to treat the general case $k\ge 1$.

\begin{prop}\label{prop:bnnk}
For $0\leq k\leq n$, we have $b_{n,n,k}=b_{n,k}$.
\end{prop}
\begin{proof}
A simple bijection readily yields $b_{n,n,k}=b_{n,k}$.
For example, a Young tableau with walls in the count of $b_{n,n,k}$ can be transformed into one in the count of $b_{n,k}$ by first flipping it vertically, then horizontally, and finally replacing each entry $i$ with $2n+k+1-i$.
\end{proof}

\begin{prop}\label{Definition-bnk-Proposition}
The sequence $b_{n,m,k}$, with domain $n\ge m-1$, $m\ge k-1$, and $k\ge -1$, is uniquely determined by the recurrence relation
\begin{equation}\label{rec:bnmk}
b_{n,m,k}=(m-k+1)b_{n,m,k-1} + b_{n,m-1,k} + b_{n-1,m,k}, \quad 0\le k\le m\le n,\quad (n,m,k)\neq (0,0,0),
\end{equation}
subject to the initial condition
$$
b_{0,0,0}=1,\quad b_{m-1,m,k}=b_{n,k-1,k}=b_{n,m,-1}=0.
$$
\end{prop}
\begin{proof}
The recurrence is well-defined: for $0\le k\le m\le n$ and $(n,m,k)\neq (0,0,0)$, each term on the right-hand side involves indices within the extended domain $n\ge m-1$, $m\ge k-1$, $k\ge -1$. The initial condition $b_{0,0,0}=1$ (from combinatorial interpretation) is necessary; without it, the recurrence would yield $b_{0,0,0}=b_{0,0,-1}+b_{0,-1,0}+b_{-1,0,0}=0$ by the boundary conditions. Note that the values $b_{1,0,0}=1$ and $b_{1,1,0}=1$ are not listed
as initial conditions because they follow from the recurrence.

For $n\ge 1$, we prove the recurrence by considering the position of the largest entry \(v=n+m+k\) in a deformed Young tableau with walls counted by \(b_{n,m,k}\).
\begin{itemize}
\item[(Case 1)] \(v\) is in the bottom row. Such an entry exists only when \(n\ge m+1\); removing it gives a tableau counted by \(b_{n-1,m,k}\). If \(n=m\), no such entry exists, which corresponds to the boundary condition \(b_{m-1,m,k}=0\).

\item[(Case 2)] \(v\) is in the middle row. Such an entry exists only when \(m\ge k+1\); removing it gives a tableau counted by \(b_{n,m-1,k}\). If \(m=k\), no such entry exists, which corresponds to the boundary condition \(b_{n,k-1,k}=0\).

\item[(Case 3)] \(v\) is in the top row. Such an entry exists only when \(k\ge 1\); removing it gives a tableau counted by $b_{n,m,k-1}$, and inserting $v$ back into the top row has $m-(k-1)$ possible positions. If $k=0$, no such entry exists, which corresponds to the boundary condition \(b_{n,m,-1}=0\).
\end{itemize}
Summing the contributions from the three cases gives the recurrence.
\end{proof}

The case $k=0$ is known: $b_{n,m,0}$ counts standard Young tableaux of shape $(n,m)$ and hence has a closed formula by the well-known hook-length formula.
If $b_{n,m,k-1}$ has a formula, then the recurrence can be used to solve for $b_{n,m,k}$. This is done in Subsection \ref{subsection-recurrence-gf} using a variation of the kernel method.
However, to prove the formula for $b_{m,m,k}$ along this line, we would need to guess an explicit formula of $b_{n,m,k}$ for general $k$ and then prove it. We quickly abandoned this idea because it fails the prime check (i.e., some coefficients contain large prime factors).

Another way of looking at the recurrence is inspired by Theorem~\ref{Semi-closed-formula-Ank} and Proposition~\ref{Definition-bnk-Proposition}.
We have a conjectured formula of $b_{m,m,k}$ from which we can conjecture formulas of $b_{m+n,m,k}$ for small $n$. But the formula
quickly becomes too complicated and also fails the prime check.

A closer look at the data reveals that the formula of $b_{m-1,m,k}$ derived from the recursion passes the prime check.
This leads to our introduction of the following variant of $b_{n,m,k}$.

\begin{dfn}[A variant of $b_{n,m,k}$]\label{Def-Omega-nmk}
We define the sequence $\omega_{n,m,k}$, with domain $n\ge -1$, $m\ge k-1$, and $k\ge -1$, by the recurrence
\begin{equation}\label{rec:omega}
\omega_{n,m,k}=\omega_{n-1,m+1,k}-(m-k+2)\omega_{n-1,m+1,k-1}-\omega_{n-2,m+1,k}
\end{equation}
for $n\ge 1$, $k\ge 0$, $m\ge k-1$, and $(n,m,k)\neq (1,-1,0)$,
subject to the boundary conditions
\[
\omega_{0,-1,0}=0,\qquad  \omega_{1,-1,0}=0, \qquad \omega_{n,m,-1}=0,\qquad \omega_{-1,m,k}=0,
\]
and the initial condition
\begin{equation}\label{omega-n=0mk}
\omega_{0,m,k}=\sum_{i=0}^k \frac{\gamma_{k-i}}{i!}\cdot \frac{2^{m-k}}{(m-k+1)!}\cdot (2m+k+i-1)!!,\qquad m\ge k-1\ge -1 \text{ and } (m,k)\neq (-1,0),
\end{equation}
where $\gamma_j$ satisfy the recurrence \eqref{Recurren-Ck0} with $\gamma_0=1$.
\end{dfn}

The above definition is well-founded: since the index $n$ on the right-hand side of \eqref{rec:omega} is strictly smaller than on the left, the recurrence determines $\omega_{n,m,k}$ recursively once the values on the initial slice $n=0$ and $n=-1$ are specified. The domain conditions are forced by the initial condition \eqref{omega-n=0mk}: if $m=k-2$, then the right-hand side has denominator $(m-k+1)!=(-1)!$, which is undefined; if $(m,k)=(-1,0)$, then the term with $i=0$ involves $(-3)!!$, which is not defined.

The boundary condition $\omega_{1,-1,0}=0$ is a technical necessity. Applying the recurrence \eqref{rec:omega} at $(n,m,k)=(1,-1,0)$ would give
\[
\omega_{1,-1,0}=\omega_{0,0,0}-\omega_{0,0,-1}-\omega_{-1,0,0}=1,
\]
since $\omega_{0,0,0}=1$ and the boundary terms vanish. This would conflict with the boundary condition $\omega_{n,k-1,k}=0$ that we aim to prove. The excluded case $(n,m,k)=(1,-1,0)$ in \eqref{rec:omega} will be shown to correspond to the initial case $(n,m,k)=(0,0,0)$ in \eqref{rec:bnmk}.

For $n\ge 1$, $k\ge 0$, $m\ge k-1$, and $(n,m,k)\neq (1,-1,0)$, it is straightforward to verify that each term on the right-hand side of \eqref{rec:omega} remains within the prescribed domain. For instance, for $\omega_{n-1,m+1,k}$, we need $n-1\ge -1$, $m+1\ge k-1$, and $k\ge -1$, all of which are immediate.

We also note from \eqref{omega-n=0mk} that
\[
\omega_{0,0,0}=1
\quad\text{and}\quad
\omega_{0,k-1,k}=\frac{1}{2}\sum_{i=0}^k \frac{\gamma_{k-i}}{i!}\, (3k+i-3)!!=0,\quad (k\geq 1)
\]
where the last equality follows from \eqref{Recurren-Ck0}.

The rest of this section is devoted to proving the following relationship, which implies
Theorem~\ref{MainConject} in view of Theorem~\ref{Semi-closed-formula-Ank}.
\begin{thm}\label{Theorem=omega=bnmk}
We have $\omega_{n,m,k}=b_{n+m,m,k}$ for $n+m\geq m \geq k\geq 0$.
\end{thm}

\subsection{Two key lemmas}\label{subsection-Key-lemmas}

The most technical part of the proof of Theorem~\ref{Theorem=omega=bnmk} lies in showing that $\omega_{n,k-1,k}=0$.
For this purpose, we establish two auxiliary lemmas, the first of which is an explicit recursion for $\omega_{n,k-1,k}$ (see Lemma~\ref{Theorem-Omega-k-1} below).
This recursion contains a family of coefficients $\alpha_s(p,q)$, whose explicit form is given in~\eqref{Equation-Alpha}.

The discovery of this closed form was guided by the prime check: we first computed the coefficients from the recursion for small parameter values and observed that, in every case, they factored into products of small primes.
This pattern gave us the confidence to conjecture the general formula, which was subsequently verified by induction.

\begin{lem}\label{Theorem-Omega-k-1}
For any $1\leq s\leq n$ and $(n,k,s)\neq (1,0,1)$, the sequence $\omega_{n,k-1,k}$ satisfies the following recursion:
\begin{align*}
\omega_{n,k-1,k}=\sum_{p=1}^{\left\lfloor \frac{s+1}{2}\right\rfloor} \sum_{q=1}^{s+2-2p} \alpha_s(p,q)\cdot \omega_{n-s-1,k+s-p,k+1-q}
-\sum_{p=1}^{\left\lceil \frac{s+1}{2}\right\rceil} \sum_{q=1}^{s+3-2p} \alpha_{s+1}(p,q)\cdot \omega_{n-s,k+s-p,k+1-q},
\end{align*}
where $p, q\geq 1$ are integers, and
\begin{align}\label{Equation-Alpha}
\alpha_{s}(p,q)=\frac{(-1)^{q-p+1}\cdot (s-1+q-p)!}{(s-q-2p+2)!\cdot (q-1)!\cdot 2^{q-1}\cdot (p-1)!}.
\end{align}
\end{lem}
\begin{proof}
The proof is by induction on $s$.
If $s=1$, then we have
\begin{align*}
\omega_{n,k-1,k}&=\alpha_1(1,1)\cdot \omega_{n-2,k,k}-\alpha_2(1,1)\omega_{n-1,k,k}-\alpha_2(1,2)\omega_{n-1,k,k-1}
\\ &=-\omega_{n-2,k,k}+\omega_{n-1,k,k}-\omega_{n-1,k,k-1}.
\end{align*}
This is consistent with the recurrence for $\omega_{n,k-1,k}$ in~\eqref{rec:omega}, hence the lemma holds for $s=1$.

For $s\ge 2$, assume inductively that the formula holds for $s-1$. We now prove it for $s$.

By the recurrence for $\omega_{n,m,k}$ in~\eqref{rec:omega}, we now have
\begin{align*}
\omega_{n,k-1,k}&=\sum_{p=1}^{\left\lfloor \frac{s}{2}\right\rfloor} \sum_{q=1}^{s+1-2p} \alpha_{s-1}(p,q)\cdot \omega_{n-s,k+s-1-p,k+1-q}
-\sum_{p=1}^{\left\lceil \frac{s}{2}\right\rceil} \sum_{q=1}^{s+2-2p} \alpha_{s}(p,q)\cdot \omega_{n-s+1,k+s-1-p,k+1-q}
\\ &=\sum_{p=1}^{\left\lfloor \frac{s}{2}\right\rfloor} \sum_{q=1}^{s+1-2p} \alpha_{s-1}(p,q)\cdot \omega_{n-s,k+s-1-p,k+1-q}
- \sum_{p=1}^{\left\lceil \frac{s}{2}\right\rceil} \sum_{q=1}^{s+2-2p} \alpha_{s}(p,q)\cdot \omega_{n-s,k+s-p,k+1-q}
\\ &\quad -\sum_{p=1}^{\left\lceil \frac{s}{2}\right\rceil} \sum_{q=1}^{s+2-2p} \alpha_{s}(p,q)\cdot \big( -(s+q-p)\omega_{n-s, k+s-p, k-q}-\omega_{n-s-1, k+s-p, k+1-q} \big)
\\ &= \sum_{p=1}^{\left\lfloor \frac{s+1}{2}\right\rfloor} \sum_{q=1}^{s+2-2p} \alpha_{s}(p,q)\cdot \omega_{n-s-1, k+s-p, k+1-q}
+\sum_{p=1}^{\left\lfloor \frac{s}{2}\right\rfloor} \sum_{q=1}^{s+1-2p} \alpha_{s-1}(p,q)\cdot \omega_{n-s,k+s-1-p,k+1-q}
\\ & \quad -\sum_{p=1}^{\left\lceil \frac{s}{2}\right\rceil} \sum_{q=1}^{s+2-2p} \alpha_{s}(p,q)\cdot \omega_{n-s,k+s-p,k+1-q}
+\sum_{p=1}^{\left\lceil \frac{s}{2}\right\rceil} \sum_{q=1}^{s+2-2p} \alpha_{s}(p,q)\cdot (s+q-p)\omega_{n-s, k+s-p, k-q},
\end{align*}
where we use the fact $\lceil\frac{s}{2}\rceil=\lfloor\frac{s+1}{2}\rfloor$ in the last equality.
Notice that all $\omega$ terms now have the common form $\omega_{n-s,k+s-p,k+1-q}$ after shifting the indices
$p$ and $q$ appropriately. The proof strategy is to show that this identity holds coefficient-wise,
i.e., the coefficients of $\omega_{n-s,k+s-p,k+1-q}$ on both sides are equal. After shifting the indices,
this reduces to the following identity:
\begin{align}
& \sum_{p=2}^{\left\lfloor \frac{s}{2}\right\rfloor+1} \sum_{q=1}^{s+3-2p} \alpha_{s-1}(p-1,q)\cdot \omega_{n-s,k+s-p,k+1-q}
 -\sum_{p=1}^{\left\lceil \frac{s}{2}\right\rceil} \sum_{q=1}^{s+2-2p} \alpha_{s}(p,q)\cdot \omega_{n-s,k+s-p,k+1-q}  \nonumber
\\ &\quad +\sum_{p=1}^{\left\lceil \frac{s}{2}\right\rceil} \sum_{q=2}^{s+3-2p} \alpha_{s}(p,q-1)\cdot (s+q-p-1)\omega_{n-s, k+s-p, k+1-q}\label{Equatio-Need-omega}
\\= &-\sum_{p=1}^{\left\lceil \frac{s+1}{2}\right\rceil} \sum_{q=1}^{s+3-2p} \alpha_{s+1}(p,q)\cdot \omega_{n-s,k+s-p,k+1-q}. \nonumber
\end{align}

We begin by verifying the case $p=1$, that is,  Eq.~\eqref{Equatio-Need-omega} reduces to
\begin{align*}
&-\sum_{q=1}^{s}\alpha_s(1,q) \omega_{n-s,k+s-1,k+1-q}+\sum_{q=2}^{s+1}\alpha_s(1,q-1)(s+q-2)\omega_{n-s,k+s-1,k+1-q}
\\ =&-\sum_{q=1}^{s+1}\alpha_{s+1}(1,q)\omega_{n-s,k+s-1,k+1-q}.
\end{align*}
This is divided into the following three subcases:
\begin{enumerate}
\item When $q=1$, the equality $-\alpha_s(1,1)=-\alpha_{s+1}(1,1)$ follows from~\eqref{Equation-Alpha}.
\item When $q=s+1$, the identity $\alpha_s(1,s)\cdot (2s-1)=-\alpha_{s+1}(1,s+1)$ holds by~\eqref{Equation-Alpha}.
\item When $2\leq q\leq s$, the equality $-\alpha_s(1,q)+\alpha_s(1,q-1)\cdot (s+q-2)=-\alpha_{s+1}(1,q)$ still follows from~\eqref{Equation-Alpha}.
\end{enumerate}

We now turn to the case $p>1$. This requires us to distinguish the parity of $s$. We begin by assuming $s$ is even.

When $p$ attains its maximum value of $\frac{s+2}{2}$, we must have $q=1$, and the identity
\[
\alpha_{s-1}\left(\frac{s}{2},1\right)=-\alpha_{s+1}\left(\frac{s+2}{2},1\right)
\]
follows directly from~\eqref{Equation-Alpha}.

For $2\leq p\leq \frac{s}{2}$, it suffices to prove~\eqref{Equation-Alph-p=2tos}.
\begin{align}
& \sum_{q=1}^{s+3-2p} \alpha_{s-1}(p-1,q)\cdot \omega_{n-s,k+s-p,k+1-q}
 -\sum_{q=1}^{s+2-2p} \alpha_{s}(p,q)\cdot \omega_{n-s,k+s-p,k+1-q} \nonumber
\\ &\quad +\sum_{q=2}^{s+3-2p} \alpha_{s}(p,q-1)\cdot (s+q-p-1)\omega_{n-s, k+s-p, k+1-q} \label{Equation-Alph-p=2tos}
\\= &- \sum_{q=1}^{s+3-2p} \alpha_{s+1}(p,q)\cdot \omega_{n-s,k+s-p,k+1-q}.\nonumber
\end{align}
The verification splits naturally into three cases according to the value of $q$:

\begin{enumerate}
\item If $q=1$, then the identity reduces to
\[
\alpha_{s-1}(p-1,1)-\alpha_s(p,1)=-\alpha_{s+1}(p,1),
\]
which is an immediate consequence of~\eqref{Equation-Alpha}.

\item If $q=s+3-2p$ (the maximum possible value), the required identity becomes
\[
\alpha_{s-1}(p-1,s+3-2p)+\alpha_s(p,s+2-2p)\cdot (2s-3p+2)=-\alpha_{s+1}(p,s+3-2p),
\]
which again follows from~\eqref{Equation-Alpha}.

\item For the interior range $2\leq q\leq s+2-2p$, the identity simplifies to
\[
\alpha_{s-1}(p-1,q)-\alpha_s(p,q)+\alpha_s(p,q-1)\cdot (s+q-p-1)=-\alpha_{s+1}(p,q),
\]
also a direct consequence of~\eqref{Equation-Alpha}.
\end{enumerate}

This completes the proof for even $s$.

Now consider the case where $s$ is odd.
The argument is entirely analogous, with only the boundary values of $p$ adjusted.
When $p=\frac{s+1}{2}$ (the maximum), we have $q\in\{1,2\}$, and the two identities
\begin{align*}
\alpha_{s-1}\left(\frac{s-1}{2},1\right)-\alpha_s\left(\frac{s+1}{2},1\right)&=-\alpha_{s+1}\left(\frac{s+1}{2},1\right),\\
\alpha_{s-1}\left(\frac{s-1}{2},2\right)+\alpha_s\left(\frac{s+1}{2},1\right)\cdot \frac{s+1}{2}
&=-\alpha_{s+1}\left(\frac{s+1}{2},2\right)
\end{align*}
both follow from~\eqref{Equation-Alpha}.
For $2\leq p\leq \frac{s-1}{2}$, the same identity~\eqref{Equation-Alph-p=2tos} holds, and the three-case verification by $q$ proceeds exactly as in the even case.
We omit the repetitive details.

We have thus established that~\eqref{Equatio-Need-omega} holds and that the formula for $\omega_{n,k-1,k}$ is valid for the current value of $s$.
This completes the induction.
\end{proof}

The proof of the following lemma is somewhat technical and relies on a separate constant-term argument.
To avoid interrupting the main proof, we defer its verification to Subsection~\ref{Proof-CT-Lemma}.

\begin{lem}\label{Lemma-Sum-Free-i}
For any $n, k,i\in \mathbb{N}$ with $i\leq k$ and $(k,i)\neq (0,0)$, we have
\begin{align*}
&\sum_{p=1}^{\left\lceil \frac{n+1}{2}\right\rceil} \sum_{q=1}^{\min\{n+3-2p,k+1-i\}} \frac{(-1)^{q-p}\cdot 2^{n-p}}{(n+3-2p-q)!\cdot (p-1)!} \cdot \binom{k-i}{q-1}\cdot \frac{(2n+4k-2p-2q+1-i)!!}{(4k-3-i)!!}
\\=&2^{n-1}\binom{n+3k-2}{n}.
\end{align*}
\end{lem}
The case $(k,i)=(0,0)$ is excluded, as it would involve the undefined term $(-3)!!$; this case is treated separately in the proof of Lemma \ref{Theorem-omega=nk-1=0}.

\subsection{Finishing the proof of Theorem~\ref{MainConject}}
\label{subsection-transform-bnkm}

In this subsection, we  prove Theorem~\ref{Theorem=omega=bnmk}, which completes the proof of Theorem~\ref{MainConject}.

\begin{lem}\label{Theorem-omega=nk-1=0}
For any $n,k\in \mathbb{N}$, we have $\omega_{n,k-1,k}=0$.
\end{lem}
\begin{proof}
We first consider the case $k=0$, where the statement is $\omega_{n,-1,0}=0$ for all $n\ge 0$. For $n=0$ and $n=1$, this is exactly the boundary conditions $\omega_{0,-1,0}=0$ and $\omega_{1,-1,0}=0$ in Definition \ref{Def-Omega-nmk}. For $n>1$, taking $s=n$ in Lemma \ref{Theorem-Omega-k-1} gives
\begin{align}\label{Equation-S=N=0-k0}
\omega_{n,-1,0}=-\sum_{p=1}^{\left\lceil \frac{n+1}{2}\right\rceil} \sum_{q=1}^{n+3-2p} \alpha_{n+1}(p,q)\cdot \omega_{0,n-p,1-q}.
\end{align}
Since $\omega_{0,n-p,1-q}=0$ unless $q=1$, we have
\begin{align*}
\omega_{n,-1,0}&=-\sum_{p=1}^{\lceil\frac{n+1}{2}\rceil}\alpha_{n+1}(p,1)\,\omega_{0,n-p,0}\\
&=-\sum_{p=1}^{\lceil\frac{n+1}{2}\rceil}\frac{(-1)^p2^{n-p}(2n-2p-1)!!}{(n+2-2p)!(p-1)!}
= -\sum_{p=1}^{\lceil\frac{n+1}{2}\rceil}\frac{(-1)^p(2n-2p)!}{(n+2-2p)!(p-1)!(n-p)!}\\
&=-\sum_{p=0}^{\lfloor\frac{n}{2}\rfloor}\frac{(-1)^{p+1}(2n-2p-2)!}{(n-2p)!p!(n-p-1)!}
= \frac{1}{n-1} \sum_{p=0}^{\lfloor\frac{n}{2}\rfloor} (-1)^p\binom{n-1}{p}\binom{2n-2p-2}{n-2}\\
&=\frac{1}{n-1} \sum_{p=0}^{n-1}(-1)^p\binom{n-1}{p}\binom{2(n-1)-2p}{n-2},
\end{align*}
where the last equality follows from the fact that the terms with $p>\lfloor n/2\rfloor$ vanish. Therefore,
\begin{align*}
\omega_{n,-1,0}&=\frac{1}{n-1} \sum_{p=0}^{n-1}(-1)^p [x^{n-2}](1+x)^{2(n-1)-2p}\binom{n-1}{p}\\
&=\frac{1}{n-1} [x^{n-2}](1+x)^{2(n-1)}\sum_{p=0}^{n-1}\binom{n-1}{p}\left(-(1+x)^{-2}\right)^p\\
&=\frac{1}{n-1} [x^{n-2}](1+x)^{2(n-1)}\left(1-(1+x)^{-2}\right)^{n-1}\\
&=\frac{1}{n-1} [x^{n-2}] (x^2+2x)^{n-1}=0.
\end{align*}
Thus $\omega_{n,-1,0}=0$ for all $n\ge 0$.

Now assume $k>0$. For $n=0$, the result $\omega_{0,k-1,k}=0$ follows from \eqref{Recurren-Ck0} and \eqref{omega-n=0mk}, as noted after Definition \ref{Def-Omega-nmk}. For $n\ge 1$, taking $s=n$ in Lemma \ref{Theorem-Omega-k-1} gives
\begin{align}\label{Equation-S=N=0}
\omega_{n,k-1,k}=-\sum_{p=1}^{\left\lceil \frac{n+1}{2}\right\rceil} \sum_{q=1}^{n+3-2p} \alpha_{n+1}(p,q)\cdot \omega_{0,k+n-p,k+1-q}.
\end{align}
By Eqs.~\eqref{Equation-Alpha} and \eqref{omega-n=0mk}, we have
\begin{small}
\begin{align*}
& \omega_{n,k-1,k}
\\=&\sum_{p=1}^{\left\lceil \frac{n+1}{2}\right\rceil} \sum_{q=1}^{n+3-2p} \frac{(-1)^{q-p}\cdot 2^{n-p}}{(n+3-2p-q)!\cdot (q-1)!\cdot (p-1)!} \sum_{i=0}^{k+1-q} \frac{\gamma_{k+1-q-i}}{i!}\cdot (2n+3k-2p-q+i)!!
\\=& \sum_{p=1}^{\left\lceil \frac{n+1}{2}\right\rceil} \sum_{q=1}^{n+3-2p} \frac{(-1)^{q-p}\cdot 2^{n-p}}{(n+3-2p-q)!\cdot (q-1)!\cdot (p-1)!} \sum_{i=0}^{k+1-q} \frac{\gamma_{i}}{(k+1-q-i)!}\cdot (2n+4k-2p-2q+1-i)!!,
\end{align*}
\end{small}
where $\gamma_{j}$ satisfy recurrence relation \eqref{Recurren-Ck0} and initial condition $\gamma_{0}=1$. 
By interchanging the order of the summation above, we obtain
\begin{align*}
\omega_{n,k-1,k}
&= \sum_{i=0}^{k} \sum_{p=1}^{\left\lceil \frac{n+1}{2}\right\rceil} \sum_{q=1}^{\min\{n+3-2p, k+1-i\}} \frac{(-1)^{q-p}\cdot 2^{n-p}}{(n+3-2p-q)!\cdot (q-1)!\cdot (p-1)!}\\
&\quad \cdot \frac{(k-i)!}{(k+1-q-i)!} \cdot\frac{ (2n+4k-2p-2q+1-i)!!}{(4k-3-i)!!}\cdot\frac{\gamma_{i}}{(k-i)!}\cdot (4k-3-i)!!.
\end{align*}
By Lemma \ref{Lemma-Sum-Free-i}, this reduces to
\begin{align*}
\omega_{n,k-1,k}
&=2^{n-1}\binom{n+3k-2}{n}\sum_{i=0}^{k}\frac{\gamma_{i}}{(k-i)!}\, (4k-3-i)!!\\
&=0,
\end{align*}
where the last equality follows from \eqref{Recurren-Ck0}. This completes the proof.
\end{proof}

We are now ready for the proof of Theorem~\ref{Theorem=omega=bnmk}.
\begin{proof}[{\bf Proof of Theorem~\ref{Theorem=omega=bnmk}}]
For convenience, we set $\tilde{\omega}_{n,m,k}=b_{n+m,m,k}$. Its domain $n\ge -1$, $m\ge k-1$, $k\ge -1$ is obtained by translating that of $b_{n,m,k}$, and coincides with the domain of $\omega_{n,m,k}$ (see Definition~\ref{Def-Omega-nmk}). By the recurrence~\eqref{rec:bnmk} for $b_{n,m,k}$, we have
$$
b_{n+m,m,k}=(m-k+1)b_{n+m,m,k-1} +b_{n+m,m-1,k}+ b_{n+m-1,m,k}
$$
with $0\le k\le m\le n+m$ and $(n,m,k)\neq (0,0,0)$.
Thus $\tilde{\omega}_{n,m,k}$ satisfies the recurrence
\begin{equation}\label{til:omega}
\tilde{\omega}_{n,m,k}= (m-k+1)\tilde{\omega}_{n,m,k-1}+\tilde{\omega}_{n+1,m-1,k}+\tilde{\omega}_{n-1,m,k}
\end{equation}
with $0\le k\le m\le n+m$ and $(n,m,k)\neq (0,0,0)$.

The initial and boundary conditions for $b_{n,m,k}$ in Proposition~\ref{Definition-bnk-Proposition} translate to the initial condition
\begin{equation}\label{ini:omega1}
\tilde{\omega}_{0,0,0}=1,
\end{equation}
and the boundary conditions
\begin{equation}\label{ini:omega2}
\tilde{\omega}_{-1,m,k}=b_{m-1,m,k}=0,\qquad
\tilde{\omega}_{n,k-1,k}=b_{n+k-1,k-1,k}=0,\qquad
\tilde{\omega}_{n,m,-1}=b_{n+m,m,-1}=0.
\end{equation}

Note that recurrence~\eqref{til:omega} for $\tilde{\omega}_{n,m,k}$, together with the initial and boundary conditions in~\eqref{ini:omega1} and~\eqref{ini:omega2}, determines $\tilde{\omega}_{n,m,k}$ uniquely; see \cite{Bousquet2000} for a general criterion for the uniqueness of solutions to such recurrences.

To finish the proof, it remains to show that $\omega_{n,m,k}$ satisfies the same recurrence relation and the same initial and boundary conditions as $\tilde{\omega}_{n,m,k}$.

Setting $n\leftarrow n+1$ in~\eqref{rec:omega} yields
$$
\omega_{n+1,m,k}=\omega_{n,m+1,k}-(m-k+2)\omega_{n,m+1,k-1}-\omega_{n-1,m+1,k},
$$
which can be recast as
$$
\omega_{n,m+1,k}=(m-k+2)\omega_{n,m+1,k-1}+\omega_{n+1,m,k}+\omega_{n-1,m+1,k}.
$$
Subsequently, setting $m\leftarrow m-1$ leads to
\begin{equation}\label{Rec:omegaTwo}
\omega_{n,m,k}=(m-k+1)\omega_{n,m,k-1}+\omega_{n+1,m-1,k}+\omega_{n-1,m,k}.
\end{equation}
The above two transformations correspond to the map $(n,m,k)\mapsto (n-1,m+1,k)$. Under this map, the excluded point $(1,-1,0)$ in the recurrence for $\omega$ maps to $(0,0,0)$, and the boundary point $(0,-1,0)$ maps to $(-1,0,0)$. Hence the range of validity for \eqref{Rec:omegaTwo} is $n\ge 0$, $k\ge 0$, $m\ge k$, and $(n,m,k)\neq (0,0,0)$, which is equivalent to the range for \eqref{til:omega}: $0\le k\le m\le n+m$ with $(n,m,k)\neq (0,0,0)$.

We note that $\tilde{\omega}_{-1,0,0}$ is not specified as an initial condition. This value is never needed, since the only recurrence that would involve it is the excluded case $(n,m,k)=(0,0,0)$ in \eqref{til:omega}.

Thus $\omega_{n,m,k}$ satisfies the same recurrence relation as $\tilde{\omega}_{n,m,k}$ in~\eqref{til:omega}.

Lemma~\ref{Theorem-omega=nk-1=0} ensures that $\omega_{n,k-1,k}=0=\tilde{\omega}_{n,k-1,k}$. Finally, the remaining initial and boundary conditions are immediate from Definition~\ref{Def-Omega-nmk}:
$$
\omega_{0,0,0}=1,\quad \omega_{-1,m,k}=0,\quad \omega_{n,m,-1}=0.
$$
This completes the proof of the theorem.
\end{proof} 

Combining Proposition~\ref{prop:bnnk} with Theorem~\ref{Theorem=omega=bnmk} yields $b_{n,k}=b_{n,n,k}=\omega_{0,n,k}$, which results in the following expression for $b_{n,k}$ in view of~\eqref{omega-n=0mk}.

\begin{thm}\label{Conjecture-Bnk--semi-closed}
Let $0\leq k\leq n$. We have
\begin{align*}
b_{n,k}=\frac{2^{n-k}}{(n-k+1)!}\cdot\sum_{i=0}^k \frac{\gamma_{k-i}}{i!}\cdot (2n+k+i-1)!!,
\end{align*}
where $\gamma_{j}$ satisfy recurrence relation \eqref{Recurren-Ck0} and initial condition $\gamma_{0}=1$.
\end{thm}

\begin{proof}[{\bf Proof of Theorem~\ref{MainConject}}]
The proof now follows from Theorem~\ref{Conjecture-Bnk--semi-closed} and Theorem~\ref{Semi-closed-formula-Ank}.
\end{proof}

\subsection{Proof of Lemma~\ref{Lemma-Sum-Free-i}}\label{Proof-CT-Lemma}
We need the following two constant term facts to prove Lemma~\ref{Lemma-Sum-Free-i}.

Let $\mathrm{CT}_x F(x)$ denote the constant term of the formal Laurent series $F(x)$. However, the product of two formal Laurent series
need not converge: consider $\sum_{m\ge 0} x^m \sum_{m\ge 0} x^{-m}$.

Instead we work in a field of iterated Laurent series. See \cite{Xin04,Xin15} for details.
In this section, all constant terms are taken in the field of iterated Laurent series
$\mathbb{R}((x))((y))((z))$. Elements in this field are Laurent series in $z$, whose coefficients are double Laurent series in $\mathbb{R}((x))((y))$
defined recursively.
The convergence issue is solved by the field structure.
Every rational function has a unique iterated Laurent series expansion, and the constant term operators with respect to different variables commute.

\begin{Fact}\label{Fact-one}
Let $i$ be a nonnegative integer. If $k$ is a nonnegative integer, possibly less than $i$, or if $k$ is a half integer, then
$$\mathop{\mathrm{CT}}_x \frac{(1+x)^k}{x^i}=
\binom{k}{i}.$$
\end{Fact}

\begin{Fact}\label{Fact-two}
If $F(x)$ only contains nonnegative powers of $x$, then
$$\CT_x  F(x)\frac{1}{1-u/x} =\CT_x F(x) \sum_{n\ge 0} \left(\frac{u}{x}\right)^n = F(u),$$
where $u$ is independent of $x$, and the displayed sum belongs to $\mathbb{R}((x))((y))((z))$.
\end{Fact}
This fact follows easily from the linearity of the constant term operator. For rigorous definitions, see \cite{Xin15}.

\begin{proof}[{\bf Proof of Lemma~\ref{Lemma-Sum-Free-i}}]
We denote the left-hand side of the equation by $\mathrm{LHS}$.
Note that the right-hand side is independent of $i$, a property not obvious for the left-hand side.

We obtain
\begin{align*}
\mathrm{LHS}&=\sum_{p=1}^{\left\lceil \frac{n+1}{2}\right\rceil} \sum_{q=1}^{\min\{n+3-2p,k+1-i\}}
(-1)^{p-q}2^{n-p}\binom{n+2-p-q}{p-1}\binom{k-i}{q-1}
\\ &\quad\quad\quad \cdot\frac{(4k-i-1)(4k-i+1)\cdots (4k-i+1+2n-2p-2q)}{(n+2-p-q)!}
\\ & =\sum_{p=1}^{\left\lceil \frac{n+1}{2}\right\rceil} \sum_{q=1}^{\min\{n+3-2p,k+1-i\}}(-1)^{p-q}2^{2n-2p-q+2}\binom{n+2-p-q}{p-1}\binom{k-i}{q-1}
\\ &\quad\quad\quad \cdot \binom{2k+n-p-q+\frac{1-i}{2}}{n+2-p-q}.
\end{align*}
By Fact \ref{Fact-one}, our problem thus reduces to computing the constant term in the following equation:
\begin{align*}
\mathrm{LHS}&=\mathop{\mathrm{CT}}_{x,y,z}\sum_{p=1}^{\left\lceil \frac{n+1}{2}\right\rceil} \sum_{q=1}^{\min\{n+3-2p,k+1-i\}}
(-1)^{p-q}2^{2n-2p-q+2} \frac{(1+x)^{k-i}}{x^{q-1}}\cdot \frac{(1+y)^{n+2-p-q}}{y^{p-1}}
\\ &\quad\quad\quad\quad\quad\quad\quad\quad\quad\quad\quad\quad  \cdot\frac{(1+z)^{2k+n-p-q+\frac{1-i}{2}}}{z^{n+2-p-q}}.
\end{align*}
At this point, extending the upper bounds of the summations is equivalent to adding some zeros by Fact \ref{Fact-one}.
Consequently, we obtain
\begin{align*}
\mathrm{LHS}&=\mathop{\mathrm{CT}}_{x,y,z}\sum_{p=1}^{\infty} \sum_{q=1}^{\infty}
(-1)^{p-q}2^{2n-2p-q+2} \frac{(1+x)^{k-i}(1+y)^{n+2-p-q}(1+z)^{2k+n-p-q+\frac{1-i}{2}}}{x^{q-1}y^{p-1}z^{n+2-p-q}}.
\end{align*}
Applying the geometric series formula yields
\begin{align*}
\mathrm{LHS}&=\mathop{\mathrm{CT}}_{x,y,z}\sum_{p=1}^{\infty} \frac{(-1)^{p+1}2^{2n-2p+1}(1+y)^{n+1-p}}{(1+x)^{i-k}y^{p-1}z^{n+1-p}(1+z)^{p-2k-n+\frac{i+1}{2}}}\cdot \frac{1}{1+\frac{z}{2x(1+y)(1+z)}}
\\ &=\mathop{\mathrm{CT}}_{x,y,z}\frac{2^{2n-1}(1+y)^n}{(1+x)^{i-k}z^n(1+z)^{-2k-n+\frac{i+3}{2}}}
\cdot\frac{1}{1+\frac{z}{4y(1+z)(1+y)}}\cdot \frac{1}{1+\frac{z}{2x(1+y)(1+z)}}.
\end{align*}
We first eliminate $x$ using Fact \ref{Fact-two}, where $u=\frac{-z}{2(1+y)(1+z)}$ (we note that $\sum_{n\ge 0} (u/x)^n$ is a valid expansion in $\mathbb{R}((x))((y))((z))$). This gives
\begin{align*}
\mathrm{LHS}&=\mathop{\mathrm{CT}}_{y,z}\frac{2^{2n-1}(1+y)^n}{z^n(1+z)^{-2k-n+\frac{i+3}{2}}\left( 1+\frac{z}{4y(1+z)(1+y)}\right) \left(1-\frac{z}{2(1+y)(1+z)}\right)^{i-k}}.
\end{align*}
Next, we eliminate $y$. Only the denominator factor $1+\frac{z}{4y(1+y)(1+z)}$ may result in negative powers of $y$. By the quadratic formula, we can write
$$\frac{1}{1+\frac{z}{4y(1+y)(1+z)}} = \frac{y(1+y)}{(y-Y_1)(y-Y_2)} = \frac{(1+y)}{(1-Y_1/y)(y-Y_2)},$$
where
$$Y_1=\frac{1-\sqrt{1+z}}{2\sqrt{1+z}} =-\frac{1}{4}z+O(z^2) \quad \text{and}\quad Y_2=\frac{-1-\sqrt{1+z}}{2\sqrt{1+z}}=-1+O(z).$$
Here $O(z)$ denotes a first-order infinitesimal.
Then applying Fact \ref{Fact-two} with $u=Y_1$ (we note that $\sum_{n\ge 0} (u/y)^n$ is a valid expansion in $\mathbb{R}((y))((z))$) gives
\begin{align*}
\mathrm{LHS}=\mathop{\mathrm{CT}}_{z} 2^{n-2}z^{-n}(1+z)^{n+2k-\frac{i+3}{2}}\frac{(1+\sqrt{1+z})^{n+1}}{(\sqrt{1+z})^n}\cdot \left( \frac{1+z+\sqrt{1+z}}{(1+z)(1+\sqrt{1+z})} \right)^{k-i}.
\end{align*}
Note that some of the above computations were performed via \texttt{Maple}.
According to
$$\frac{1+z+\sqrt{1+z}}{(1+z)(1+\sqrt{1+z})}=\frac{\sqrt{1+z}}{1+z}=\frac{1}{\sqrt{1+z}},$$
we obtain
\begin{align*}
\mathrm{LHS} &=\mathop{\mathrm{CT}}_{z} 2^{n-2}z^{-n}(1+z)^{n+2k-\frac{i+3}{2}}\frac{(1+\sqrt{1+z})^{n+1}}{(\sqrt{1+z})^n}\cdot \left( \frac{1}{\sqrt{1+z}} \right)^{k-i}
\\ &=2^{n-2}\cdot\mathop{\mathrm{CT}}_{z} \frac{(1+z)^{\frac{3k+n-3}{2}} (1+\sqrt{1+z})^{n+1}}{z^n}.
\end{align*}
It is noteworthy that the above expression is independent of $i$.
We assume the reader is familiar with the knowledge of residues.
Informally, the residue can be viewed as the coefficient of the $z^{-1}$ term in a Laurent series expansion.
Therefore, we have
\begin{align*}
\mathrm{LHS} &=2^{n-2}\cdot\mathop{\mathrm{Res}}_{z=0} \frac{(1+z)^{\frac{3k+n-3}{2}} (1+\sqrt{1+z})^{n+1}}{z^{n+1}}
\\ &= 2^{n-2}\cdot \frac{1}{2\pi i}\oint_{|z|=\varepsilon} \frac{(1+z)^{\frac{3k+n-3}{2}} (1+\sqrt{1+z})^{n+1}}{z^{n+1}} dz,
\end{align*}
where the integration path is taken near $z=0$, with $\varepsilon$ sufficiently small.
Let $z=4t+4t^2$. Then we have $dz=4(1+2t)dt$ and
\begin{align*}
\mathrm{LHS} &= 2^{n-2}\cdot \frac{1}{2\pi i}\oint_{|t|=\varepsilon} \frac{(1+2t)^{3k+n-3}(2(1+t)^{n+1})}{(4t(1+t))^{n+1}} \cdot 4(1+2t) dt
\\&=2^{n-2}\cdot \frac{1}{2\pi i}\oint_{|t|=\varepsilon} \frac{(1+2t)^{3k+n-2}}{2^{n-1}t^{n+1}} dt
\\ &= 2^{-1}\cdot\mathop{\mathrm{Res}}_{t=0} \frac{(1+2t)^{3k+n-2}}{t^{n+1}}
=2^{-1}\cdot\mathop{\mathrm{CT}}_{t} \frac{(1+2t)^{3k+n-2}}{t^{n}}
\\ &= 2^{n-1}\binom{3k+n-2}{n}.
\end{align*}
This completes the proof.
\end{proof}

\subsection{Applications to the enumeration of tree-child networks}
\label{apps:thm1.1}

We now present several corollaries concerning the enumeration of tree-child networks.
Pons and Batle~\cite[Propositions 8-12]{PB} derived the following results, Corollaries \ref{PonsBatlePropo8}, \ref{PonsBatlePropo9}, \ref{PonsBatlePropo10}, \ref{PonsBatlePropo11}, and \ref{PonsBatlePropo12} under the assumption that Conjecture \ref{Pons-Batle} holds.
We can now remove this assumption.

\begin{cor}\label{PonsBatlePropo8}
The cardinalities $|\mathcal{TC}_{n,k}|$ satisfy
$$(n-k)|\mathcal{TC}_{n,k}|=(n+1-k)(n-k) |\mathcal{TC}_{n,k-1}|+n(2n+k-3) |\mathcal{TC}_{n-1,k}|,$$
where the initial values are $|\mathcal{TC}_{1,0}|=1$ and $|\mathcal{TC}_{i,-1}|=|\mathcal{TC}_{i,i}|=0$ for any $i$.
\end{cor}
\begin{proof}
By combining Eqs.~\eqref{eq:conPB} and~\eqref{rec:ank}, we obtain this result.
\end{proof}

\begin{cor}\label{PonsBatlePropo9}
The sequence $|\mathcal{TC}_{n,k}|$ satisfies
$$(n-k)!\cdot |\mathcal{TC}_{n,k}|=\sum_{i=0}^k n(2n+i-3)(n-1-i)!\cdot |\mathcal{TC}_{n-1,i}|,$$
where the initial values are $|\mathcal{TC}_{1,0}|=1$ and $|\mathcal{TC}_{i,-1}|=|\mathcal{TC}_{i,i}|=0$ for any $i$.
\end{cor}
\begin{proof}
By recurrence~\eqref{rec:ank}, we obtain
$$a_{n+1,k}=\sum_{i=0}^k (2n+i+1)a_{n,i}.$$
The result then follows immediately from the relation above combined with Eq.~\eqref{eq:conPB}.
\end{proof}

\begin{cor}\label{PonsBatlePropo10}
The sequence  $|\mathcal{TC}_{n,k}|$ satisfies
$$|\mathcal{TC}_{k+m+1,k}|=\sum_{\ell=0}^m (\ell+2)\left[\prod_{i=\ell+1}^m \left( 1+ \frac{k}{i+1}\right)(2i+3k-1) \right] |\mathcal{TC}_{k+\ell+1,k-1}|$$
for $k\geq 1$.
\end{cor}
\begin{proof}
By Eq.~\eqref{eq:recank2}, we obtain
$$a_{n,k+1}=\sum_{i=k+1}^n \prod_{j=i+1}^n (2j+k)a_{i,k}.$$
This result follows from the above expression and Eq.~\eqref{eq:conPB}.
\end{proof}

\begin{cor}\label{PonsBatlePropo11}
Let $\delta_j=j!\cdot \gamma_j$ for $j\geq 0$.
The sequence $|\mathcal{TC}_{n,k}|$ satisfies
$$|\mathcal{TC}_{n,k}|=\binom{n}{k}\sum_{i=0}^k \binom{k}{i} (2n+2k-i-3)!!\cdot \delta_i,$$
where the coefficients $\delta_i$ are determined by the recursion
$$\delta_i=-\sum_{j=1}^i \binom{i}{j} \frac{(2i+j-3)!!}{(3i-3)!!} \delta_{i-j},\quad \text{and}\quad \delta_0=1.$$
\end{cor}
\begin{proof}
This is a consequence of Theorem~\ref{Semi-closed-formula-Ank} together with Eq.~\eqref{eq:conPB}.
\end{proof}

Pons and Batle \cite[Proposition 7]{PB} determined the asymptotic behavior of $a_{n,k}$ as $n\rightarrow \infty$ with fixed $k$, and then, assuming Conjecture \ref{Pons-Batle}, derived the following result \cite[Proposition 12]{PB}.
Subsequently, Fuchs, Huang, and Yu \cite{FHY22} also derived the main term of the following asymptotic formula.

\begin{cor}\label{PonsBatlePropo12}
For $k$ fixed and $n\rightarrow \infty$, the sequence $|\mathcal{TC}_{n,k}|$ grows as
\begin{align*}
|\mathcal{TC}_{n,k}|&=\binom{n}{k}\frac{\sqrt{2}}{e^n}(2n)^{n+k-1}\bigg(1-\sqrt{\frac{\pi}{2}} k (2n)^{-\frac{1}{2}} +\frac{14k^2-26k+11}{12} (2n)^{-1}
\\&- \sqrt{\frac{\pi}{2}} k \frac{31k^2-93k+70}{48}(2n)^{-\frac{3}{2}} + \frac{2900k^4-14376k^3+25264k^2-19332 k +5565}{6048} (2n)^{-2}
\\& \quad +O(n^{-\frac{5}{2}})\bigg).
\end{align*}
\end{cor}

\section{Generating functions for $b_{n,k}$ and $b_{n,m,k}$}\label{Section-Generating-function}

Introduce the generating functions for $b_{n,k}$ and $b_{n,m,k}$ as
\begin{align*}
D_k(t)=\sum_{n\geq 0} b_{n,k} t^n\quad\text{and}\quad B_k(x,t)=\sum_{k\leq m\leq n} b_{n,m,k} x^{n-m} t^m.
\end{align*}
In this section, we obtain an explicit formula for $D_k(t)$ from Theorem~\ref{Conjecture-Bnk--semi-closed} and derive a system of functional equations for $B_k(x,t)$ using the kernel method \cite{Bousquet2000}.
A nice application of the kernel method to lattice paths is presented in \cite{C.Banderier2002}.

\subsection{An explicit formula for $D_k(t)$}\label{subsection-explicite-gf}

\begin{thm}\label{Generating-function-odd-line-bnk}
Let $k\geq 1$. The generating function for $b_{n,k}$ is given by
\begin{align*}
D_k(t)&=\sum_{j=0}^{\frac{k-1}{2}} \frac{\gamma_{k-2j-1}}{(2j+1)! \cdot 2^{j+\frac{3k+1}{2}}}\cdot \left(j+\frac{3k-1}{2}\right)! \cdot \binom{2j+3k-1}{j+\frac{3k-1}{2}}\cdot \frac{t^{k-1}}{(1-4t)^{j+\frac{3k}{2}}}
\\ & \quad\quad +\sum_{j=0}^{\frac{k-1}{2}}\frac{\gamma_{k-2j}\cdot  2^{j+\frac{3k-5}{2}}}{(2j)!}\cdot\left(j+\frac{3k-3}{2}\right)!\cdot \frac{t^{k-1}}{(1-4t)^{j+\frac{3k-1}{2}}}
\end{align*}
when $k$ is odd, and
\begin{align*}
D_k(t)&=\sum_{j=0}^{\frac{k}{2}} \frac{\gamma_{k-2j}}{(2j)! \cdot 2^{j+\frac{3k}{2}}}\cdot \left(j+\frac{3k-2}{2}\right)! \cdot \binom{2j+3k-2}{j+\frac{3k-2}{2}}\cdot \frac{t^{k-1}}{(1-4t)^{j+\frac{3k-1}{2}}}
\\ & \quad\quad +\sum_{j=0}^{\frac{k}{2}-1}\frac{\gamma_{k-2j-1}\cdot  2^{j+\frac{3k-4}{2}}}{(2j+1)!}\cdot\left(j+\frac{3k-2}{2}\right)!\cdot \frac{t^{k-1}}{(1-4t)^{j+\frac{3k}{2}}}
\end{align*}
when $k$ is even.
\end{thm}
\begin{proof}
By Theorem \ref{Conjecture-Bnk--semi-closed}, we have
\begin{align}\label{Equation-overline-bnk}
b_{n,k}=\sum_{i=0}^k \frac{\gamma_{k-i}}{i!}\cdot \frac{2^{n-k}}{(n-k+1)!}\cdot (2n+k+i-1)!!,
\end{align}
where $\gamma_{j}$ satisfy the recursion~\eqref{Recurren-Ck0} and the initial condition $\gamma_{0}=1$. We need to derive the formula for $D_k(t)$ according to the parity of $k$.

We begin with the case for $k$ being odd.
It is convenient to rewrite~\eqref{Equation-overline-bnk} as
\begin{equation}\label{eq:T12}
b_{n,k}=\sum_{j=0}^{\left\lfloor\frac{k}{2}\right\rfloor} \frac{\gamma_{k-2j}}{(2j)!}\cdot \frac{(2n+k+2j-1)!!}{(n-k+1)!} 2^{n-k}
+\sum_{j=0}^{\left\lceil\frac{k}{2} \right\rceil-1} \frac{\gamma_{k-2j-1}}{(2j+1)!}\cdot \frac{(2n+k+2j)!!}{(n-k+1)!} 2^{n-k}.
\end{equation}
When $k$ is odd and $0\leq j\leq \frac{k-1}{2}$, we first consider the following generating function:
\begin{align*}
T_1:=\sum_{n\geq 0} \frac{(2n+k+2j)!!}{(n-k+1)!}\cdot 2^{n-k} t^n =\sum_{n\geq 0} \frac{2^{-j-\frac{3k+1}{2}}\cdot (2n+k+2j+1)!}{(n+j+\frac{k+1}{2})!\cdot (n-k+1)!} t^n.
\end{align*}
Setting $m=n+j+\frac{k+1}{2}$ gives
\begin{align*}
T_1&=\sum_{m\geq j+\frac{k+1}{2}} \frac{2^{-j-\frac{3k+1}{2}}\cdot (2m)!}{m!\cdot (m-(j+\frac{3k-1}{2}))!} t^{m-j-\frac{k+1}{2}}
\\& =2^{-j-\frac{3k+1}{2}}\cdot t^{-j-\frac{k+1}{2}} \sum_{m\geq j+\frac{k+1}{2}} \binom{2m}{m} \binom{m}{j+\frac{3k-1}{2}} \cdot \left(j+\frac{3k-1}{2}\right)!\cdot t^m.
\end{align*}
By introducing  a variable $z$ and extracting the coefficient of $z^{j+\frac{3k-1}{2}}$, we have
\begin{align*}
T_1 &=2^{-j-\frac{3k+1}{2}}\cdot \left(j+\frac{3k-1}{2}\right)! \cdot t^{-j-\frac{k+1}{2}}
\sum_{m\geq j+\frac{k+1}{2}} \binom{2m}{m} [z^{j+\frac{3k-1}{2}}] (1+z)^mt^m
\\ &= 2^{-j-\frac{3k+1}{2}}\cdot \left(j+\frac{3k-1}{2}\right)! \cdot t^{-j-\frac{k+1}{2}}[z^{j+\frac{3k-1}{2}}]
\sum_{m\geq 0} \binom{2m}{m}((1+z)t)^m.
\end{align*}
Since $\sum_{s\geq 0}\binom{2s}{s}t^s=\frac{1}{\sqrt{1-4t}}$, we obtain
\begin{align*}
T_1 &= 2^{-j-\frac{3k+1}{2}}\cdot \left(j+\frac{3k-1}{2}\right)! \cdot t^{-j-\frac{k+1}{2}}[z^{j+\frac{3k-1}{2}}]\frac{1}{\sqrt{1-4(1+z)t}}
\\ &=2^{-j-\frac{3k+1}{2}}\cdot \left(j+\frac{3k-1}{2}\right)! \cdot t^{-j-\frac{k+1}{2}}[z^{j+\frac{3k-1}{2}}] \frac{1}{\sqrt{1-4t}} \cdot \frac{1}{\sqrt{1-\frac{4tz}{1-4t}}}.
\end{align*}
Once again, using the generating function of the central binomial coefficients, we have
\begin{align*}
T_1 &=2^{-j-\frac{3k+1}{2}}\cdot \left(j+\frac{3k-1}{2}\right)! \cdot t^{-j-\frac{k+1}{2}}\frac{1}{\sqrt{1-4t}} [z^{j+\frac{3k-1}{2}}] \cdot \sum_{s\geq 0}\binom{2s}{s}\left(\frac{tz}{1-4t}\right)^s
\\ &=2^{-j-\frac{3k+1}{2}}\cdot \left(j+\frac{3k-1}{2}\right)! \cdot t^{-j-\frac{k+1}{2}}\frac{1}{\sqrt{1-4t}} \binom{2j+3k-1}{j+\frac{3k-1}{2}} \frac{t^{j+\frac{3k-1}{2}}}{(1-4t)^{j+\frac{3k-1}{2}}},
\end{align*}
which can be simplified to
\begin{equation}\label{eq:T1}
T_1=2^{-j-\frac{3k+1}{2}}\cdot \left(j+\frac{3k-1}{2}\right)! \cdot \binom{2j+3k-1}{j+\frac{3k-1}{2}} \frac{t^{k-1}}{(1-4t)^{j+\frac{3k}{2}}}.
\end{equation}
Secondly, we consider the following generating function:
\begin{align*}
T_2:=\sum_{n\geq 0} \frac{(2n+k+2j-1)!!}{(n-k+1)!}\cdot 2^{n-k} t^n
=\sum_{n\geq 0} \frac{2^{n+j+\frac{k-1}{2}}\cdot (n+j+\frac{k-1}{2})!\cdot 2^{n-k}}{(n-k+1)!} t^n.
\end{align*}
According to $(1-t)^{-\alpha}=1+\sum_{s\geq 1} \frac{\alpha (\alpha +1)\cdots (\alpha +s-1)}{s!}t^s$,
setting $m=n-k+1$ and $k=2p+1$ (since $k$ is odd) yields
\begin{align*}
T_2 &=\sum_{m\geq 0} \frac{2^{m+j+3p}\cdot (m+j+3p)!\cdot 2^{m-1}}{m!} t^{m+2p}
=2^{j+3p-1}t^{2p} \sum_{m\geq 0} \frac{(3p+m+j)!}{m!} (4t)^{m}
\\ &= 2^{j+3p-1}t^{2p} \sum_{m\geq 0} \frac{(3p+j)!\cdot (3p+j+1)(3p+j+2)\cdots (3p+j+m)}{m!} (4t)^{m}
\\ &= 2^{j+3p-1}(3p+j)! t^{2p}\cdot \frac{1}{(1-4t)^{3p+j+1}}.
\end{align*}
Therefore, we have
\begin{equation}\label{eq:T2}
T_2= 2^{j+\frac{3k-5}{2}}\left(j+\frac{3k-3}{2}\right)! \frac{t^{k-1}}{(1-4t)^{j+\frac{3k-1}{2}}}.
\end{equation}
Combining~\eqref{eq:T12},~\eqref{eq:T1} and~\eqref{eq:T2} establishes the desired formula for $D_k(t)$ for odd $k$.

Similarly, we can derive the  generating function for $D_k(t)$ when $k$ is even. The details of the proof will be omitted, as the argument is analogous.
\end{proof}

\subsection{Functional equations for $B_k(x,t)$ via the kernel method}
\label{subsection-recurrence-gf}

We aim to derive a system of functional equations for $B_k(x,t)$ using Proposition~\ref{Definition-bnk-Proposition} and the kernel method.
This is an independent approach for $D_k(t)$.
Let $F_0(x,t)=1$ and
$$ F_k(x,t):=\sum_{ k\leq m\leq n} (m-k+1)b_{n,m,k-1} x^{n-m}t^m$$
for $k\geq1$.
For $k\geq 1$, we have
\begin{align*}
F_k(x,t)&=\sum_{0\leq k\leq m\leq n} m\cdot b_{n,m,k-1} x^{n-m}t^m +\sum_{0\leq k\leq m\leq n} (1-k)\cdot b_{n,m,k-1} x^{n-m}t^m
\\ &= t\cdot \frac{\partial B_{k-1}(x,t)}{\partial t}+(1-k)\cdot B_{k-1}(x,t).
\end{align*}
By the recurrence of $b_{n,m,k}$ in Proposition~\ref{Definition-bnk-Proposition}, we obtain
\begin{align*}
B_k(x,t)&=F_k(x,t) +\sum_{0\leq k\leq m\leq n} b_{n,m-1,k} x^{n-m}t^m + \sum_{0\leq k\leq m\leq n} b_{n-1,m,k} x^{n-m}t^m
\\ &=F_k(x,t) + \frac{t}{x}\cdot B_k(x,t)-\frac{t}{x}\cdot D_k(t) +x\cdot B_k(x,t).
\end{align*}
Therefore, we have
\begin{equation}\label{eq:kernel}
\left(1-x-\frac{t}{x}\right)B_k(x,t)=F_k(x,t)-\frac{t}{x}\cdot D_k(t).
\end{equation}
This is a special type of functional equation with two unknowns $B_k(x,t)$ and $D_k(t)$, where
the factor $\left( 1-x-\frac{t}{x}\right)$ on the left-hand side is called the \emph{kernel} of~\eqref{eq:kernel}.
Solving the kernel $1-x-\frac{t}{x}=0$, i.e., $x^2-x+t=0$ for $x$ gives two solutions:
$$X_1(t)=\frac{1+\sqrt{1-4t}}{2}=1+O(t) \quad \text{and} \quad X_2(t)=\frac{1-\sqrt{1-4t}}{2}=t\cdot C(t).$$
According to the kernel method \cite{Bousquet2000}, we can substitute $x=X_2(t)$ to obtain
\begin{align*}
D_k(t)=\left(\frac{x}{t} \cdot F_k(x,t) \right)\Big|_{x=X_2(t)}.
\end{align*}
The above functional equations are summarized in the following theorem.
\begin{thm}\label{Genereating-function-Dkt}
For $k\geq 1$, we have the following system of functional equations
\begin{align*}
\left\{
\begin{aligned}
& D_k(t)=\left(\frac{x}{t} \cdot F_k(x,t) \right)\Big|_{x=X_2(t)}, \\
& F_k(x,t)= t\cdot \frac{\partial B_{k-1}(x,t)}{\partial t}+(1-k)\cdot B_{k-1}(x,t),\\
& B_k(x,t)=\frac{F_k(x,t)-\frac{t}{x}\cdot D_k(t)}{1-x-\frac{t}{x}},
\end{aligned}
\right.
\end{align*}
with the initial condition $F_0(x,t)=1$, $D_0(t)=C(t)=\frac{1-\sqrt{1-4t}}{2t}$, and $B_0(x,t)=\frac{1}{1-x-X_2(t)}$.
\end{thm}

From a computational perspective, Theorem \ref{Genereating-function-Dkt} is highly practical to compute expressions for $D_k(t)$ and $B_k(x,t)$.
In the computational implementation, we perform the substitution $t=u(1-u)$, thus ensuring that we consistently work with rational functions and avoid fractional exponents. At this time, we have
$$X_2(t)=u=\frac{1-\sqrt{1-4t}}{2},\quad B_0(x,t)=\frac{-1}{x-(1-u)}.$$
and
$$\frac{\partial B_{k-1}(x,t)}{\partial t}=\frac{1}{1-2u}\cdot \frac{\partial B_{k-1}(x,u(1-u))}{\partial u}.$$
Using Theorem \ref{Genereating-function-Dkt}, we can easily compute expressions for $D_k(t)$ for $k\leq 140$ by \texttt{Maple}, which agree with the results from Theorem~\ref{Generating-function-odd-line-bnk}. This provided strong computational evidence for Theorem~\ref{MainConject} before we found a proof.

\section{Recurrence for $b_{n,k}$ via linear extensions of some posets}
\label{sec:LExten}

In our initial approach to Theorem~\ref{MainConject}, we began by computing the number of linear extensions of some posets associated with $b_{n,k}$.
As a byproduct of this work, we provide formulas for the linear extensions of some posets.
We ultimately obtain an exceptionally lengthy recurrence relation for $b_{n,k}$ as presented below.

\begin{thm}\label{bnk-formula-Recu-three}
Let $n$ be a positive integer and $0\leq k\leq n$. Then we have
\begin{small}
\begin{equation*}
b_{n,k}=\binom{2n+k}{n}\cdot f_{n,k}-\sum_{j=1}^n \sum_{s=0}^k \sum_{m=0}^s
\frac{(j+k-s)\cdot(n-j-m)!\cdot(k+2j-m-1)!}{2^{k-s}\cdot (j-k+s)!\cdot(k-s)!\cdot j!\cdot (s-m)!\cdot(n-j-s)!}\cdot b_{n-j,m},
\end{equation*}
where $f_{n,k}=\frac{1}{2^k\cdot k!}\frac{(n+k)!}{(n-k)!}$.
\end{small}
\end{thm}

The proof of Theorem~\ref{bnk-formula-Recu-three} is given in Appendix~\ref{App:B}; the outline of the proof steps are:
\begin{enumerate}
\item We present closed formulas for linear extensions of several poset families related to $b_{n,k}$; see Lemmas \ref{fnk-closed-formula}, \ref{Fnk-Directsum-Ln}, \ref{Linear-Ex-tildeF}.

\item We study a counting function $u_{n,k}$ associated with $b_{n,k}$, as defined in Section \ref{Character-bnk-unk}.
We derive two transformation formulas between $b_{n,k}$ and $u_{n,k}$; see Theorems \ref{BNK-UNK} and \ref{UNK-To-BNK}.
This is necessary to derive the recurrence relation for $b_{n,k}$.

\item By applying the inclusion-exclusion principle together with the previous lemmas, we obtain a recurrence relation for $b_{n,k}$; see Theorems~\ref{bnk-formula-two} and~\ref{bnk-formula-Recu-three}.
\end{enumerate}

The approach developed here, which counts linear extensions of posets arising from combinatorial structures with local constraints and applies inclusion-exclusion to remove forbidden patterns, is not specific to the Young tableaux with walls considered in this paper.
Indeed, many enumeration problems in algebraic combinatorics and phylogenetics can be reformulated as counting linear extensions of suitably defined posets, where the constraints are encoded by directed edges and the exclusion of undesirable configurations is naturally handled by inclusion-exclusion.
We expect this framework to be applicable to other families of tableaux with boundary conditions, as well as to counting problems on directed trees and networks where local order constraints appear.

\section{Concluding remarks, open problems}
\label{Section-concluding-remark}

This paper investigates enumeration formulas for two specific types of Young tableaux with walls.
The main result is the proof of Theorem~\ref{MainConject}, which leads to the verification of Pons and Batle's enumerative  conjecture on counting phylogenetic tree-child networks (see Conjecture~\ref{conj:PB}).
The enumeration of Young tableaux with walls is generally a challenging problem.
This paper provides several methods and techniques to facilitate future related studies.

For future research, we highlight the following two open problems.
\begin{prob}
Are there combinatorial proofs of Theorem~\ref{MainConject}, Conjecture~\ref{conj:PB} and Corollary~\ref{Canject-bnk}?
\end{prob}

\begin{prob}
Can Theorem~\ref{bnk-formula-Recu-three} be applied to prove Theorem~\ref{MainConject} or Corollary~\ref{Canject-bnk}?
\end{prob}

During the review process, we became aware of the independent work of Liu, Wallner, and Yu \cite{Liu-Wallner-Yu}, which also employs generating function techniques to verify Conjecture~\ref{Pons-Batle} for $k\leq 250$.
They further introduce a lattice-path enumeration framework that provides a combinatorial interpretation for some of the generating functions.
Such a framework is indeed natural from the perspective of our ternary recurrence; however, since our analytic approach already suffices to establish the desired identities, we do not pursue this direction here.
Their model is based on a slightly different ternary recurrence, offering a complementary viewpoint.


\noindent\textbf{Conflict of interest:} The authors declare that they have no conflict of interest.




\section*{Acknowledgments}
The authors thank the anonymous referees for their careful reading of the manuscript and for their constructive suggestions, which improved the presentation of this paper.
We are grateful to Wenjie Fang for communicating Theorem~\ref{MainConject} (Fuchs and Yu's enumerative conjecture)  to us during his visit to Shandong University in the summer of 2025, to Shu Xiao Li for insightful discussions on Proposition~\ref{Definition-bnk-Proposition}, and to Michael Fuchs and Guan-Ru Yu for kindly sharing the intriguing background information, as well as for their inspiring comments, regarding  Theorem~\ref{MainConject}.
This work was supported by the National Natural Science Foundation of China (grants
12571355, 12322115 \& 12271301) and the Fundamental Research Funds for
the Central Universities.
Feihu Liu was partially supported by the Postdoctoral Fellowship Program and China Postdoctoral Science Foundation (Grant No. BX2026002).

\newpage
\appendix

\section{Appendix: Formulas for $a_{n,k}$ and $b_{n,k}$ for some fixed small $k$}
We present formulas for $a_{n,k}$ and $b_{n,k}$ here for observation of their structures.

We derived the following  formulas for $a_{n,k}$ directly from recursion~\eqref{eq:recank2}.
\begin{prop}\label{proposition-ank-init-term}
Let $0\leq k\leq n$.
We derive formulas for $a_{n,k}$ for some small $k$:
\begin{align*}
&a_{n,0}=(2n-1)!!,
\\& a_{n,1}=(2n+1)!!-(2n)!!,
\\& a_{n,2}=\left(n+\frac{5}{3}\right)\cdot (2n+1)!!-(2n+2)!!,
\\& a_{n,3}=\frac{n+3}{3}\cdot(2n+3)!!-\left(n+\frac{79}{48} \right)\cdot (2n+2)!!,
\\& a_{n,4}=\left(\frac{n^2}{6}+\frac{7n}{6}+\frac{319}{189} \right)\cdot (2n+3)!!-\frac{(16n+31)(n+2)}{24}\cdot (2n+2)!!,
\\& a_{n,5}=\left(\frac{63n^2+609n+1006}{1890}\right)\cdot (2n+5)!!-\left(\frac{1}{6} n^{2}+\frac{13}{16} n +\frac{9107}{9216}\right)\cdot (2n+4)!!.
\end{align*}
\end{prop}

We derived the following  formulas for $b_{n,k}$ directly from the recursion in Theorem~\ref{bnk-formula-Recu-three}.
\begin{prop}\label{conjecture-bnk-0--5}
Let $0\leq k\leq n$.
We derive formulas for $b_{n,k}$ for some small $k$:
\begin{align*}
&b_{n,0}=\frac{1}{n+1}\binom{2n}{n},
\\& b_{n,1}=\frac{1}{2}\cdot\frac{(2n+1)!}{(n!)^2} -2^{2n-1},
\\& b_{n,2}=\frac{n(n+\frac{5}{3})}{4}\cdot \frac{(2n+1)!}{(n!)^2}-n(n+1)\cdot 2^{2n-1},
\\& b_{n,3}=\frac{(n+3)}{3\cdot 2^4}\cdot \frac{(2n+3)!}{(n-2)!(n+1)!}-(n-1)n(n+1)\left(n+\frac{79}{48}\right)\cdot 2^{2n-2},
\\& b_{n,4}=\left(\frac{n^2}{6}+\frac{7n}{6}+\frac{319}{189}\right)\cdot\frac{(2n+3)!}{2^5\cdot(n-3)!(n+1)!}
\\ & \quad\quad\quad \quad\quad\quad -\frac{(16n+31)}{3}\cdot (n-2)(n-1)n(n+1)(n+2)\cdot 2^{2n-6}.
\\& b_{n,5}=\left(\frac{1}{15} n^{3}+\frac{73}{90} n^{2}+\frac{5057}{1890} n +\frac{503}{189}\right)\cdot\frac{(2n+4)!}{2^7\cdot (n-4)!(n+2)!}
\\& \quad\quad\quad -\left(\frac{1}{6} n^{2}+\frac{13}{16} n +\frac{9107}{9216}\right)\cdot (n-3)(n-2)(n-1)n(n+1)(n+2)\cdot 2^{2n-3}.
\end{align*}
\end{prop}

\section{Appendix: Linear extensions of some posets}
\label{App:B}

In this Appendix, we provide formulas for the linear extensions of some posets.
Subsection~\ref{Section-P-partitions-linear extensions} introduces the theory of $P$-partitions and linear extensions of posets, which is required later.
Subsection~\ref{Section-a-class-of-posets} presents formulas  for linear extensions of several poset families. This provides some necessary lemmas for deriving the recurrence formula for $b_{n,k}$.
The relationship between $b_{n,k}$ and $u_{n,k}$ is characterized in Subsection~\ref{Character-bnk-unk}.
In Subsection~\ref{Section-Recurrence-for-bnk}, we establish a recurrence relation for $b_{n,k}$.

\subsection{$P$-partitions and linear extensions}\label{Section-P-partitions-linear extensions}

In this subsection, we introduce some concepts and useful lemmas  regarding $P$-partitions and linear extensions.

Let $(P,\preceq)$ be a finite \emph{partially ordered set} (or \emph{poset}).
We say that $t$ \emph{covers} $s$ if $s\prec t$ and there is no $c\in P$ such that $s\prec c\prec t$.
The \emph{Hasse diagram} of a finite poset $P$ is a graph whose vertices are the elements of $P$, whose edges represent cover relations, and in which, if $s\prec t$, then $t$ is drawn above $s$ in the diagram.

A \emph{$P$-partition} is a map $\sigma: P\rightarrow \mathbb{N}$ satisfying the following condition:
If $s\prec t$ in $P$, then $\sigma(s)\geq \sigma(t)$; that is, $\sigma$ is \emph{order-reversing}.
In addition, if $\sum_{s \in P} \sigma(s) = n$, we say $\sigma$ is a \emph{$P$-partition of $n$} and write $|\sigma| = n$.
It is often convenient to represent the relation $s \prec t$ by an arrow $s \to t$.
In this notation, a $P$-partition corresponds to a vertex labeling $\sigma$ such that for every arrow $s \to t$, we have $\sigma(s) \geq \sigma(t)$.

The concept of $P$-partitions was first investigated by MacMahon \cite[Subsections 439,441]{MacMahon19}.
Later, Knuth \cite{Knuth70} offered a clarification of MacMahon's results in his 1970 work on the theory of $P$-partitions.
The first systematic treatment of this subject, however, was established by Stanley in his seminal papers \cite{Stanley-Order-parti71,Stanley-Order-parti72}.

We denote the set of all $P$-partitions on the poset $P$ as $\pi(P)$.
The fundamental generating function associated with $\pi(P)$ is defined as
$$F_P(x_1,x_2,\ldots,x_p)=\sum_{\sigma\in \pi(P)}x_1^{\sigma(1)}x_2^{\sigma(2)}\cdots x_p^{\sigma(p)}.$$
By setting $x_i=q$ for all $i$, we obtain the \emph{enumerative generating function} of $\pi(P)$:
\begin{align*}
G_P(q)=F_P(q,q,\ldots,q)=\sum_{\sigma\in \pi(P)} q^{|\sigma|}.
\end{align*}

The set $[p]:=\{1,2,\ldots,p\}$ with its usual order forms a $p$-element chain, that is, a poset in which all elements are totally ordered. This poset is denoted by $\mathbf{p}$.
Let $\#P=p$. An order-preserving bijection $\tau: P\rightarrow \mathbf{p}$ is called a \emph{linear extension} of $P$.
The number of linear extensions of $P$ is denoted by $e(P)$.

Now, let $(P,\preceq)$ be a finite poset on the set $[p]$.
Any linear extension $\tau \colon P \to \mathbf{p}$ can be identified with the permutation $\tau^{-1}(1), \tau^{-1}(2), \dots, \tau^{-1}(p)$ of $[p]$.
The collection of all such $e(P)$ permutations is denoted by $\mathcal{L}(P)$ and is called the \emph{Jordan--H\"older set} of $P$.
Counting linear extensions of a poset is a computationally difficult problem, and was shown by Brightwell and Winkler \cite{Brightwell91} to be \#P-complete.

Let $\mu=\mu_1\mu_2\cdots \mu_p$ be a permutation of $[p]$. For $1\leq i\leq p-1$, if $\mu_i>\mu_{i+1}$, then $i$ is called a \emph{descent} of $\mu$.
The \emph{descent set} of $\mu$ is defined as
$$D_{\mu} = \{ i \in [p-1] : \mu_i > \mu_{i+1} \}.$$
Let $d(\mu)$ denote the number of descents of $\mu$, that is, $d(\mu) = \# D_{\mu}$.

\begin{thm}[\text{\cite[Theorem 3.15.5]{Stanley-Vol-1}}]
Let $(P,\preceq)$ be a finite naturally labeled poset on $[p]:=\{1,2,\ldots,p\}$. Then
\begin{align*}
F_P(x_1,x_2,\ldots, x_p)=\sum_{\mu\in \mathcal{L}(P)}\frac{\prod_{j\in D_{\mu}}x_{\mu_1}x_{\mu_2}\cdots x_{\mu_j}}{\prod_{i=1}^p (1-x_{\mu_1}x_{\mu_2}\cdots x_{\mu_i})}.
\end{align*}
\end{thm}

Define the \emph{major index} $\mathrm{maj}(\mu)$ of $\mu$ by $\mathrm{maj}(\mu)=\sum_{j\in D_{\mu}} j$.
\begin{thm}[\text{\cite[Theorem 3.15.7]{Stanley-Vol-1}}]
Let $(P,\preceq)$ be a finite naturally labeled poset on $[p]$. Then
\begin{align*}
G_P(q)=\frac{\sum_{\mu\in \mathcal{L}(P)} q^{\mathrm{maj}(\mu)}}{(1-q)(1-q^2)\cdots (1-q^p)}.
\end{align*}
\end{thm}

The following result provides a formula for the number $e(P)$ of linear extensions of a poset.
\begin{lem}[\text{\cite[Section 3, page 294]{Stanley-Vol-1}}]\label{Formula-ep-EC1}
Let $(P,\preceq)$ be a finite naturally labeled poset on $[p]$. Then
\begin{align}
e(P)=p!\cdot \bigg(\Big((1-q)^p\cdot G_P(q) \Big)\Big|_{q=1}\bigg).\label{linearExtension}
\end{align}
\end{lem}

The \emph{direct sum} of $P_1$ and $P_2$ is the poset $P_1+ P_2$ on the disjoint union $P_1 \uplus P_2$ such that $s\preceq t$ in $P_1+ P_2$ if (i) $s,t\in P_1$ and $s\preceq t$ in $P_1$, or (ii) $s,t\in P_2$ and $s\preceq t$ in $P_2$.
The \emph{ordinal sum} of posets $P_1$ and $P_2$ is the poset $P_1\oplus P_2$ on the disjoint union $P_1 \uplus P_2$ such that $s\preceq t$ in $P_1 \oplus P_2$ if (i) $s,t\in P_1$ and $s\preceq t$ in $P_1$, or (ii) $s,t\in P_2$ and $s\preceq t$ in $P_2$, or (iii) $s\in P_1$ and $t\in P_2$.

\begin{lem}[\text{\cite[Section 3.5]{Stanley-Vol-1}}]\label{Lem-Direct-sum-ep}
If $P=P_1+P_2+\cdots +P_k$ and $n_i=\#P_i$, then
\begin{align*}
e(P)=\binom{n_1+n_2+\cdots +n_k}{n_1,n_2,\ldots,n_k} e(P_1)e(P_2)\cdots e(P_k).
\end{align*}
If $P=P_1\oplus P_2\oplus \cdots \oplus P_k$ and $n_i=\#P_i$, then
\begin{align*}
e(P)=e(P_1)e(P_2)\cdots e(P_k).
\end{align*}
\end{lem}

Let $I_m$ be the poset consisting of $m$ isolated vertices, i.e., an $m$-element antichain.
Then we have $e(I_m)=m!$.
Let $L_m=I_1\oplus I_1\oplus \cdots \oplus I_1$ ($m$ times). Thus $L_m$ is an $m$-element chain. It follows that $e(L_m)=1$.

Let $P$ be a poset with $p$ elements.
In its Hasse diagram $\Gamma$, the relation $s \prec t$ is conventionally represented by an arrow $s \to t$.
If there exists a directed path from vertex $v$ to vertex $u$ in $\Gamma$, we say that $v$ \emph{reaches} $u$ and denote by $v \dashrightarrow u$.
The number of vertices in $\Gamma$ that can reach vertex $u$ is denoted by $\omega(u)$.

When the Hasse diagram $\Gamma$ of the poset $P$ is a directed tree in which every vertex has out-degree either $1$ or $0$, the enumerative generating function for $\pi(P)$ admits a simple form.

\begin{lem}[\text{\cite[Exercise 20, Page 70]{KnuthBook}}]\label{LTX-Outdegree-10}
If the Hasse diagram $\Gamma$ of the poset $P$ is a directed tree in which every vertex has out-degree either $1$ or $0$, then the generating function $G_P(q)$ is given by
\begin{align*}
G_P(q)=\prod_{v\in \Gamma}\frac{1}{1 - q^{\omega(v)}}.
\end{align*}
Furthermore, by Lemma \ref{Formula-ep-EC1}, we obtain
\begin{align}\label{DKnuthFormual}
e(P)=\frac{p!}{\prod_{v\in \Gamma} \omega(v)}.
\end{align}
\end{lem}

The formula \eqref{DKnuthFormual} is the classical \emph{hook length formula for trees}, often attributed to Knuth \cite[Exercise 20, Page 70]{KnuthBook};
the preceding $q$-analogue, while not stated explicitly in that exercise, is a straightforward extension and is included here for completeness.

\subsection{Linear extensions of a class of posets}\label{Section-a-class-of-posets}

In this subsection, we present a closed formula for the linear extensions of a class of posets.
This provides some necessary lemmas for deriving the recurrence formula for $b_{n,k}$.

Let $n$ be a positive integer.
Given a positive integer sequence $1\leq i_1<i_2<\cdots <i_k\leq n$, define the poset $F_{(n; i_1,i_2,\ldots,i_k)}$ as shown in Fig.~\ref{Tu4}, where the relation $s\prec t$ is conventionally represented by an arrow $s\to t$.
We set $F_{(n; i_1,i_2,\ldots,i_k)}=L_n$ ($n$-element chain) if $k=0$ by convention.

\begin{figure}[htp]
\centering
\includegraphics[width=0.9\linewidth]{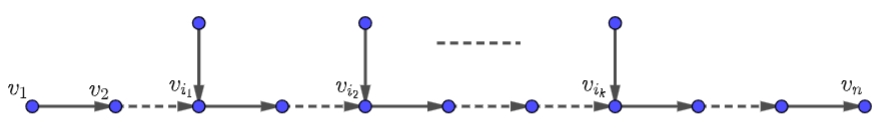}
\caption{Illustration of the poset $F_{(n; i_1,i_2,\ldots,i_k)}$.}
\label{Tu4}
\end{figure}

We now define $f_{n,k}$ as the sum of the number of linear extensions of all posets $F_{(n; i_1,i_2,\ldots,i_k)}$, i.e.,
\begin{equation*}
f_{n,k}:=\sum_{1\leq i_1<i_2<\cdots <i_k\leq n} e(F_{(n; i_1,i_2,\ldots,i_k)}).
\end{equation*}

\begin{lem}\label{fnk-formula}
Let $n$ be a positive integer and $0\leq k\leq n$. Then we have $f_{n,0}=1$ and
\begin{equation*}
f_{n,k}=\sum_{i_1=1}^{n-k+1}\sum_{i_2=i_1+1}^{n-k+2}\cdots \sum_{i_k=i_{k-1}+1}^{n} i_1(i_2+1)(i_3+2)\cdots (i_k+k-1)\quad \text{for}\quad 1\leq k\leq n.
\end{equation*}
\end{lem}
\begin{proof}
For an $n$-element chain, we know that $e(L_n)=1$. It follows that $f_{n,0}=1$.
Now we consider the cases $1\leq k\leq n$.
Given a positive integer sequence $1\leq i_1<i_2<\cdots <i_k\leq n$, the Hasse diagram of the poset $F_{(n; i_1,i_2,\ldots,i_k)}$ is a directed tree in which every vertex has out-degree either $1$ or $0$.
By definition of $\omega(v)$, we obtain
$$\omega(v_{i_1})=i_1+1,\quad \omega(v_{i_2})=i_2+2, \quad\ldots,\quad \omega(v_{i_k})=i_k+k,$$
and so
$$\prod_{v\in F_{(n; i_1,i_2,\ldots,i_k)}} \omega(v)=\frac{(n+k)!}{i_1(i_2+1)(i_3+2)\cdots (i_k+k-1)}.$$
By Lemma \ref{LTX-Outdegree-10}, we obtain
\begin{align*}
e(F_{(n; i_1,i_2,\ldots,i_k)})=\frac{(n+k)!}{(n+k)!}\cdot i_1(i_2+1)(i_3+2)\cdots (i_k+k-1).
\end{align*}
This completes the proof.
\end{proof}

\begin{lem}\label{fnk-closed-formula}
Let $n$ be a positive integer and $0\leq k\leq n$. Then we have
\begin{equation}\label{eq:fnkclose}
f_{n,0}=1,\quad f_{n,k}=\frac{1}{2^k\cdot k!}\frac{(n+k)!}{(n-k)!} \quad \text{for}\quad 1\leq k\leq n,
\end{equation}
and
\begin{equation}
\sum_{n\geq 0}f_{n,k}x^n=\frac{(2k-1)!!\cdot x^k}{(1-x)^{2k+1}}.
\end{equation}
\end{lem}
\begin{proof}
We know that
\begin{equation*}
f_{n,k}=\sum_{i_1=1}^{n-k+1}\sum_{i_2=i_1+1}^{n-k+2}\cdots \sum_{i_k=i_{k-1}+1}^{n} i_1(i_2+1)(i_3+2)\cdots (i_k+k-1)\quad \text{for}\quad 1\leq k\leq n.
\end{equation*}
The above equation can be written equivalently as
\begin{equation*}
f_{n,k}=\sum_{c_1\geq 1;\ c_k\leq n+k-1 \atop c_{i}-c_{i-1}\geq 2 \text{ for } 2\leq i\leq k} c_1c_2\cdots c_k.
\end{equation*}
We consider the following two cases according to $c_k$:
First, when $c_k=n+k-1$, we have $c_{k-1}\leq c_k-2=n+k-3$ and
\begin{align*}
f_{n,k}&=(n+k-1)\sum_{c_1\geq 1;\ c_{k-1}\leq n+k-3 \atop c_{i}-c_{i-1}\geq 2 \text{ for } 2\leq i\leq k-1} c_1c_2\cdots c_{k-1}
\\ &=(n+k-1)\cdot f_{n-1,k-1}.
\end{align*}
Second, when $c_k\neq n+k-1$, we get $f_{n,k}=f_{n-1,k}$. Therefore, we obtain
\begin{equation*}
f_{n,k}=f_{n-1,k}+ (n+k-1)\cdot f_{n-1,k-1}.
\end{equation*}

To understand the behavior of $f_{n,k}$, we first computed small values using the recurrence and applied the prime check.
The data suggested that for each fixed $k$, $f_{n,k}$ is a polynomial in $n$ of degree $2k$ with leading coefficient $1/(2^k k!)$.
Once this polynomial pattern was recognized, we were able to identify the closed form
\begin{equation*}
f_{n,k}=\frac{(n-k+1)(n-k+2)\cdots (n+k)}{2^k\cdot k!}
      =\frac{1}{2^k\cdot k!}\cdot \frac{(n+k)!}{(n-k)!}.
\end{equation*}
The verification is a straightforward induction: defining $g_{n,k}$ as the right-hand side above, one checks that $g_{n,k}$ satisfies the same recurrence
\[
g_{n,k}=g_{n-1,k}+(n+k-1)g_{n-1,k-1}
\]
and the same initial conditions as $f_{n,k}$.
Hence $f_{n,k}=g_{n,k}$ by uniqueness.

Now we have
\begin{align*}
\sum_{n\geq 0}f_{n,k}x^n&=\sum_{n\geq k}\frac{(n-k+1)(n-k+2)\cdots (n+k)}{2^k\cdot k!} x^n
\\ &= \frac{(2k)!}{k!\cdot 2^k} \sum_{n\geq k}\frac{(n+k)(n+k-1)\cdots (n-k+1)}{(2k)!}x^{n}
\\ &= \frac{(2k)!}{k!\cdot 2^k} \sum_{n\geq k}\binom{n+k}{2k}x^{n}=\frac{(2k)!}{k!\cdot 2^k} \sum_{m\geq 0}\binom{2k+m}{m}x^{m+k}
\\ &= \frac{(2k-1)!!\cdot x^k}{(1-x)^{2k+1}}.
\end{align*}
This completes the proof.
\end{proof}

Note that the sequences $(f_{n,1})_{n\geq 1}$,  $(f_{n,2})_{n\geq 2}$ and $(f_{n,3})_{n\geq 3}$ are the sequences A000217, A050534, and  A240440 in the OEIS~\cite{Sloane23}, respectively.

We now consider the linear extensions of the direct sum of the poset $F_{(n; i_1,i_2,\ldots,i_k)}$ and an $n$-element chain $L_n$.
\begin{lem}\label{Fnk-Directsum-Ln}
Let $n$ be a positive integer and $0\leq k\leq n$. Then we have
\begin{align*}
\sum_{1\leq i_1<i_2<\cdots <i_k\leq n} e(F_{(n; i_1,i_2,\ldots,i_k)}+L_n)=\binom{2n+k}{n}\cdot f_{n,k}.
\end{align*}
\end{lem}
\begin{proof}
This result follows directly from Lemma \ref{Lem-Direct-sum-ep}.
\end{proof}

Given a positive integer sequence $1\leq i_1<i_2<\cdots <i_k\leq n$, define the poset $\widetilde{F}_{(n; i_1,i_2,\ldots,i_k)}$ as shown in Fig.~\ref{Tu5}, where the relation $s\prec t$ is conventionally represented by an arrow $s\to t$.

\begin{figure}[htp]
\centering
\includegraphics[width=0.9\linewidth]{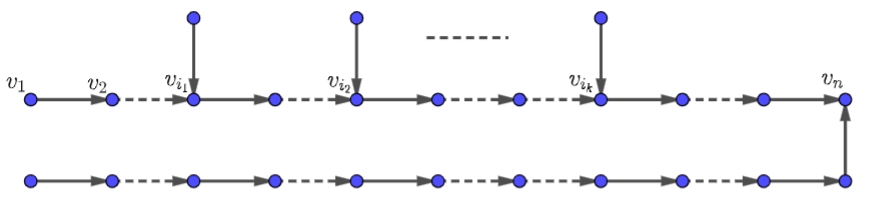}
\caption{Illustration of the poset $\widetilde{F}_{(n; i_1,i_2,\ldots,i_k)}$.}
\label{Tu5}
\end{figure}

Similarly, we define $\widetilde{f}_{n,k}$ as the sum of the number of linear extensions of all posets $\widetilde{F}_{(n; i_1,i_2,\ldots,i_k)}$, i.e.,
\begin{align*}
\widetilde{f}_{n,k}=\sum_{1\leq i_1<i_2<\cdots <i_k\leq n} e(\widetilde{F}_{(n; i_1,i_2,\ldots,i_k)}).
\end{align*}

\begin{lem}\label{Linear-Ex-tildeF}
Let $n$ be a positive integer and $0\leq k\leq n$. Then we have
\begin{align*}
\widetilde{f}_{n,k}=\binom{2n+k-1}{n}\cdot f_{n,k}.
\end{align*}
\end{lem}
\begin{proof}
It is clear that $\omega(v_n)=2n+k$. When $k=0$, by Lemma \ref{LTX-Outdegree-10}, we have
$$\widetilde{f}_{n,0}=\frac{(2n)!}{n!(n-1)!(2n)}=\binom{2n-1}{n}.$$
When $1\leq k\leq n$, by Lemma \ref{LTX-Outdegree-10}, we obtain
\begin{align*}
\widetilde{f}_{n,k}=\frac{(2n+k)!}{n!(n+k-1)!(2n+k)}\cdot\sum_{i_1=1}^{n-k+1}\sum_{i_2=i_1+1}^{n-k+2}\cdots \sum_{i_k=i_{k-1}+1}^{n} i_1(i_2+1)(i_3+2)\cdots (i_k+k-1).
\end{align*}
This completes the proof.
\end{proof}

\subsection{Characterization of $b_{n,k}$}\label{Character-bnk-unk}

Let $n$ be a positive integer.
Given a positive integer sequence $1\leq i_1<i_2<\cdots <i_k\leq n$, define the poset $D_{(n; i_1,i_2,\ldots,i_k)}$ as shown in Fig.~\ref{Tu1}, where the relation $s \prec t$ is conventionally represented by an arrow $s \to t$.
\begin{figure}[htp]
\centering
\includegraphics[width=0.9\linewidth]{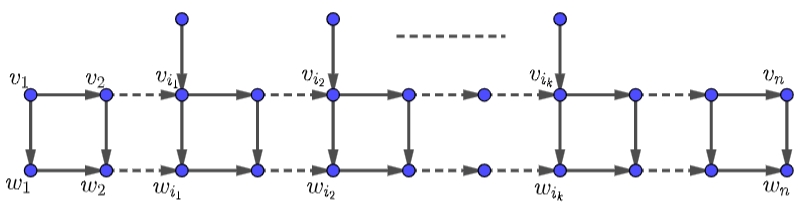}
\caption{Illustration of the poset $D_{(n; i_1,i_2,\ldots,i_k)}$.}
\label{Tu1}
\end{figure}

By the definition of $e(D_{(n; i_1,i_2,\ldots,i_k)})$, we see that this number is also equal to the number of labelings of the  poset with integers  $1,2,\ldots,2n+k$, such that every row and column is increasing along the direction of the arrow.
Therefore, we obtain
\begin{align*}
b_{n,k}=\sum_{1\leq i_1<i_2<\cdots <i_k\leq n} e(D_{(n; i_1,i_2,\ldots,i_k)}).
\end{align*}

Similarly, define the poset $U_{(n; i_1,i_2,\ldots,i_k)}$ as shown in Fig.~\ref{Tu2}, where the relation $s \prec t$ is conventionally represented by an arrow $s \to t$. In other words, $U_{(n; i_1,i_2,\ldots,i_k)}$ is obtained by reversing the $k$ downward arrows in the upper layer in $D_{(n; i_1,i_2,\ldots,i_k)}$ to upward arrows.
\begin{figure}[htp]
\centering
\includegraphics[width=0.9\linewidth]{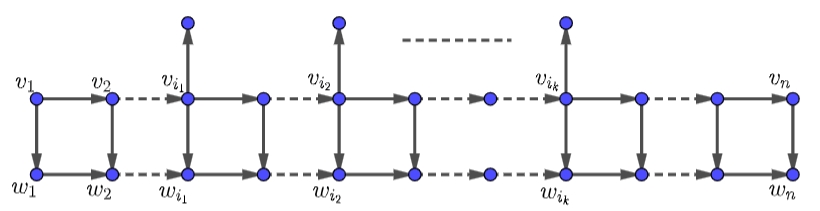}
\caption{Illustration of the poset  $U_{(n; i_1,i_2,\ldots,i_k)}$.}
\label{Tu2}
\end{figure}

We now define $u_{n,k}$ as the sum of the number of linear extensions of all posets $U_{(n; i_1,i_2,\ldots,i_k)}$, i.e.,
\begin{align*}
u_{n,k}=\sum_{1\leq i_1<i_2<\cdots <i_k\leq n} e(U_{(n; i_1,i_2,\ldots,i_k)}).
\end{align*}

We now derive a transformation formula between $b_{n,k}$ and $u_{n,k}$.
\begin{thm}\label{BNK-UNK}
Let $n$ be a positive integer and $0\leq k\leq n$. Then we have
\begin{align*}
b_{n,k}=\sum_{i=0}^k (-1)^i \binom{2n+k}{k-i}\binom{n-i}{k-i}\cdot (k-i)!\cdot u_{n,i}.
\end{align*}
\end{thm}
\begin{proof}
This formula is derived by applying the inclusion-exclusion principle and Lemma \ref{Lem-Direct-sum-ep}.
We first illustrate the above formula with an example before proving the general case.

Let $n=4$, $k=2$, and $i_1=2$, $i_2=3$. Figure~\ref{Tu3} illustrates the inclusion-exclusion decomposition of $D_{(4;2,3)}$.
\begin{figure}[htp]
\centering
\includegraphics[width=0.9\linewidth]{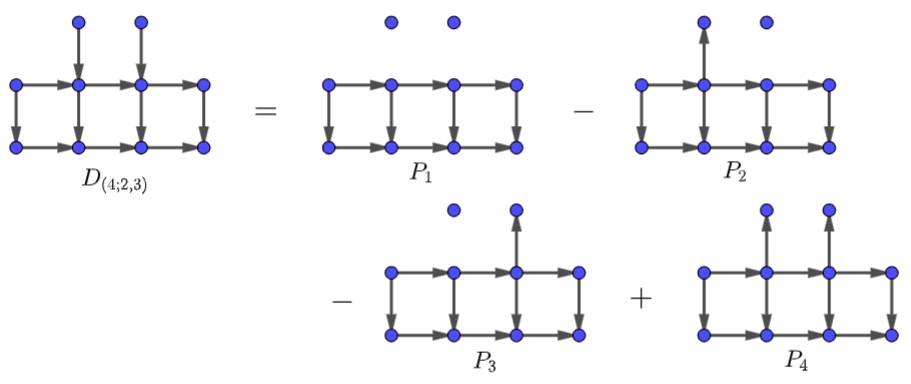}
\caption{An example of the inclusion-exclusion decomposition of the Hasse diagram of $D_{(4;2,3)}$.}
\label{Tu3}
\end{figure}

All Hasse diagrams in the figure share the same underlying set of vertices; they differ only in the orientation of two arrows at the top row, each of which may be downward, upward, or absent (the latter indicating incomparability).
The target poset $D_{(4;2,3)}$ is the special case where all arrows point downward.
Inclusion-exclusion then isolates this downward-only configuration by starting with the all-absent case $P_1=I_2+U_{(4;0)}$, subtracting those with one upward arrow ($P_2$ and $P_3$), and adding back those with two upward arrows ($P_4$):
\[
e(D_{(4;2,3)})=e(P_1)-\bigl(e(P_2)+e(P_3)\bigr)+e(P_4).
\]
Here $P_2=I_1+U_{(4;2)}$, $P_3=I_1+U_{(4;3)}$, and $P_4=U_{(4;2,3)}$.
This is precisely the reflection principle: the upward arrows correspond to reflections, and the alternating sum cancels all configurations with upward arrows, leaving only the downward-only configuration $D_{(4;2,3)}$.

The above discussion can be generalized to the general case.
Given a positive integer sequence $1\leq i_1<i_2<\cdots <i_k\leq n$, by an application of the inclusion-exclusion principle, we obtain
\begin{align*}
e(D_{(n; i_1,i_2,\ldots,i_k)})&=\sum_{j=0}^k (-1)^j \sum_{(h_1,\ldots,h_j)\subseteq (i_1,\ldots,i_k)} e(I_{k-j}+U_{(n; h_1,\ldots,h_j)})
\\ &=\sum_{j=0}^k (-1)^j \sum_{(h_1,\ldots,h_j)\subseteq (i_1,\ldots,i_k)} \binom{2n+k}{k-j} e(I_{k-j}) e(U_{(n; h_1,\ldots,h_j)})
\\ &=\sum_{j=0}^k (-1)^j \binom{2n+k}{k-j} (k-j)! \sum_{(h_1,\ldots,h_j)\subseteq (i_1,\ldots,i_k)} e(U_{(n; h_1,\ldots,h_j)}).
\end{align*}
The summation of both sides over all integer sequences $1\leq i_1<i_2<\cdots <i_k\leq n$ yields
\begin{align*}
b_{n,k}&=\sum_{j=0}^k (-1)^j \binom{2n+k}{k-j} (k-j)! \sum_{1\leq i_1<i_2<\cdots <i_k\leq n} \sum_{(h_1,\ldots,h_j)\subseteq (i_1,\ldots,i_k)} e(U_{(n; h_1,\ldots,h_j)})
\\ &= \sum_{j=0}^k (-1)^j \binom{2n+k}{k-j} (k-j)! \sum_{1\leq h_1<\cdots <h_j\leq n} \binom{n-j}{k-j} e(U_{(n; h_1,\ldots,h_j)})
\\ &= \sum_{j=0}^k (-1)^j \binom{2n+k}{k-j} (k-j)! \binom{n-j}{k-j} u_{n,j}.
\end{align*}
This completes the proof.
\end{proof}

By Theorem \ref{BNK-UNK}, we get the following result.
\begin{cor}
Let $n$ be a positive integer and $0\leq k\leq n$. Then we have
\begin{align*}
(-1)^ku_{n,k}=b_{n,k}-\sum_{i=0}^{k-1} (-1)^i \binom{2n+k}{k-i}\binom{n-i}{k-i}\cdot (k-i)!\cdot u_{n,i}.
\end{align*}
\end{cor}

A proof similar to that of Theorem \ref{BNK-UNK} gives the following result.
\begin{thm}\label{UNK-To-BNK}
Let $n$ be a positive integer and $0\leq k\leq n$. Then we have
\begin{align*}
u_{n,k}=\sum_{i=0}^k (-1)^i \binom{2n+k}{k-i}\binom{n-i}{k-i}\cdot (k-i)!\cdot b_{n,i}.
\end{align*}
\end{thm}

\subsection{Recurrence for $b_{n,k}$}\label{Section-Recurrence-for-bnk}

Let $n$ be a positive integer.
Given a positive integer sequence
$$1\leq i_1<\cdots <i_{k-s}\leq j< i_{k-s+1}<\cdots <i_k\leq n,$$ define the poset $R_{(n; i_1,\ldots,i_{k-s},j,i_{k-s+1},\ldots,i_k)}$ as shown in Fig.~\ref{Tu6}, where the relation $s \prec t$ is conventionally represented by an arrow $s \to t$.
In Fig.~\ref{Tu6}, among the downward arrows from the upper layer, there are $s$ to the right of vertex $v_j$ and $k-s$ to the left.

From an alternative perspective, the Hasse diagram of the poset $R_{(n; i_1,\ldots,i_{k-s},j,i_{k-s+1},\ldots,i_k)}$ is formed by connecting those of posets $\widetilde{F}_{(j; i_1,i_2,\ldots,i_{k-s})}$ and $D_{(n-j; i_{k-s+1}-j,i_{k-s+2}-j,\ldots,i_{k}-j)}$ with a directed edge $v_j\to v_{j+1}$. For convenience, we denote $R_{(n; i_1,\ldots,i_{k-s},j,i_{k-s+1},\ldots,i_k)}$ by
$$\widetilde{F}_{(j; i_1,i_2,\ldots,i_{k-s})}\rightarrow D_{(n-j; i_{k-s+1}-j,i_{k-s+2}-j,\ldots,i_{k}-j)}.$$

\begin{figure}[htp]
\centering
\includegraphics[width=0.95\linewidth]{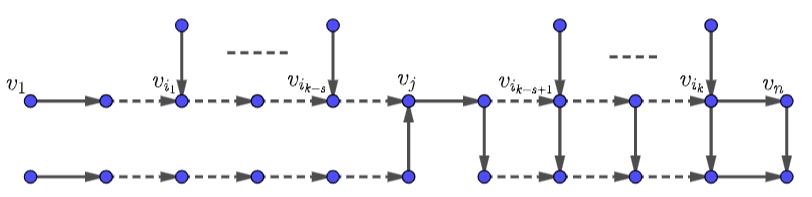}
\caption{Illustration  of the poset $R_{(n; i_1,\ldots,i_{k-s},j,i_{k-s+1},\ldots,i_k)}$.}
\label{Tu6}
\end{figure}

We define $r_{n,k}$ as the sum of the number of linear extensions of all posets $R_{(n; i_1,\ldots,i_{k-s},j,i_{k-s+1},\ldots,i_k)}$, i.e.,
\begin{align*}
r_{n,k}=\sum_{j=1}^n \sum_{s=0}^k \sum_{1\leq i_1<\cdots <i_{k-s}\leq j \atop j< i_{k-s+1}<\cdots <i_k\leq n} e(R_{(n; i_1,\ldots,i_{k-s},j,i_{k-s+1},\ldots,i_k)}).
\end{align*}

\begin{lem}\label{Rnk-js}
Let $n$ be a positive integer and $0\leq k\leq n$. Then we have
\begin{align*}
r_{n,k}=\sum_{j=1}^n \sum_{s=0}^k \binom{2j+k-s-1}{j}\cdot f_{j,k-s}\cdot \left(\sum_{i=0}^s(-1)^i\binom{2n+k}{s-i}\binom{n-j-i}{s-i}(s-i)!\cdot u_{n-j,i}\right).
\end{align*}
\end{lem}
\begin{proof}
We know that the Hasse diagram of the poset $R_{(n; i_1,\ldots,i_{k-s},j,i_{k-s+1},\ldots,i_k)}$ is formed by connecting those of posets $\widetilde{F}_{(j; i_1,i_2,\ldots,i_{k-s})}$ and $D_{(n-j; i_{k-s+1}-j,i_{k-s+2}-j,\ldots,i_{k}-j)}$ with a directed edge $v_j\to v_{j+1}$.
By Theorem \ref{BNK-UNK}, the Hasse diagram of the poset $D_{(n-j; i_{k-s+1}-j,i_{k-s+2}-j,\ldots,i_{k}-j)}$ can be decomposed using the inclusion-exclusion principle. Similarly, for the Hasse diagram of poset $R_{(n; i_1,\ldots,i_{k-s},j,i_{k-s+1},\ldots,i_k)}$, we can retain the $\widetilde{F}_{(j; i_1,i_2,\ldots,i_{k-s})}$ portion and apply the same decomposition to its latter part.
For each decomposed Hasse diagram, its linear extensions can be computed by employing Lemma \ref{Lem-Direct-sum-ep}.

For fixed $s$ and $j$, Lemma \ref{Linear-Ex-tildeF} and Theorem \ref{BNK-UNK} yield
\begin{align*}
&\sum_{1\leq i_1<\cdots <i_{k-s}\leq j< i_{k-s+1}<\cdots <i_k\leq n} e(R_{(n; i_1,\ldots,i_{k-s},j,i_{k-s+1},\ldots,i_k)})
\\=&\left(\binom{2j+k-s-1}{j}\cdot f_{j,k-s}\right)\cdot \left(\sum_{i=0}^s(-1)^i\binom{2n+k}{s-i}\binom{n-j-i}{s-i}(s-i)!\cdot u_{n-j,i}\right).
\end{align*}
This completes the proof.
\end{proof}

With Lemma~\ref{Rnk-js} in hand, we now derive the explicit formula for $b_{n,k}$.
The idea is to start with the forward posets $F_{(n; i_1,\ldots,i_k)}+L_n$ and subtract those linear extensions that contain upward arrows; these are precisely the extensions enumerated by the $r_{n,k}$ terms.

\begin{thm}\label{bnk-formula-two}
Let $n$ be a positive integer and $0\leq k\leq n$. Then we have
\begin{align*}
b_{n,k}&=\binom{2n+k}{n}\cdot f_{n,k}
\\&\quad-\sum_{j=1}^n \sum_{s=0}^k \binom{2j+k-s-1}{j}\cdot f_{j,k-s}\cdot \left(\sum_{i=0}^s(-1)^i\binom{2n+k}{s-i}\binom{n-j-i}{s-i}(s-i)!\cdot u_{n-j,i}\right).
\end{align*}
\end{thm}
\begin{proof}
Fix a positive integer sequence $1\leq i_1<i_2<\cdots <i_k\leq n$.
By definition, $e(D_{(n; i_1,\ldots,i_k)})$ counts the labelings of the poset depicted in Figure~\ref{Tu1} by the distinct integers $1,2,\ldots,2n+k$, such that every row and column is increasing along the direction of the arrows.

Consider the larger poset $F_{(n; i_1,\ldots,i_k)}+L_n$, whose Hasse diagram has no arrows in the $n$ middle positions.
For counting linear extensions, each of these $n$ unconstrained positions can be independently assigned either a downward or an upward arrow, yielding $2^n$ possible configurations.
The configuration with all downward arrows is precisely $D_{(n; i_1,\ldots,i_k)}$.
Thus, to count $D_{(n; i_1,\ldots,i_k)}$, we start with $F_{(n; i_1,\ldots,i_k)}+L_n$ and subtract those configurations that contain at least one upward arrow.

We classify the configurations containing upward arrows by the position $j$ of the rightmost upward arrow.
For a fixed $j$, the $j$-th position is upward, all positions to its right are downward, and the positions to its left are unconstrained.
Summing over all possible orientations of these left positions reduces them to the no-arrow state, which yields precisely the $R$-poset configuration illustrated in Fig.~\ref{Tu6}.

Therefore, we obtain
\begin{align*}
b_{n,k}=\sum_{1\leq i_1<\cdots <i_k\leq n} e(F_{(n; i_1,\ldots,i_k)}+L_n)-\sum_{j=1}^n \sum_{s=0}^k \sum_{\substack{1\leq i_1<\cdots <i_{k-s}\leq j \\ j< i_{k-s+1}<\cdots <i_k\leq n}} e(R_{(n; i_1,\ldots,i_{k-s},j,i_{k-s+1},\ldots,i_k)}).
\end{align*}
The theorem follows from Lemmas~\ref{Fnk-Directsum-Ln} and~\ref{Rnk-js}.
\end{proof}

We can now prove Theorem~\ref{bnk-formula-Recu-three} by using Lemma \ref{fnk-closed-formula}, Theorem \ref{UNK-To-BNK} and Theorem \ref{bnk-formula-two}.

\begin{proof}[{\bf Proof of Theorem~\ref{bnk-formula-Recu-three}}]
According to Lemma \ref{fnk-closed-formula}, Theorem \ref{UNK-To-BNK}, and Theorem \ref{bnk-formula-two}, it follows that
\begin{align*}
b_{n,k}&=\binom{2n+k}{n}\cdot f_{n,k}-\sum_{j=1}^n \sum_{s=0}^k \binom{2j+k-s-1}{j}\cdot f_{j,k-s}\cdot
\\ & \quad\quad\quad\quad \cdot \Bigg(\sum_{i=0}^s(-1)^i\binom{2n+k}{s-i}\binom{n-j-i}{s-i}(s-i)!\cdot
\\ & \quad\quad\quad\quad\quad\quad\quad\quad \cdot \sum_{m=0}^i(-1)^m\binom{2(n-j)+i}{i-m}\binom{n-j-m}{i-m}(i-m)!\cdot b_{n-j,m}\Bigg),
\end{align*}
where $f_{n,k}$ has closed formula given in~\eqref{eq:fnkclose}.
Therefore, we obtain
\begin{small}
\begin{align*}
b_{n,k}&=\binom{2n+k}{n}\cdot f_{n,k}-\sum_{j=1}^n \sum_{s=0}^k \sum_{i=0}^s \sum_{m=0}^i  \binom{2j+k-s-1}{j}\cdot f_{j,k-s}(-1)^{i+m}\cdot
\\ & \quad \cdot \Bigg(\binom{2n+k}{s-i}\binom{n-j-i}{s-i}(s-i)! \cdot\binom{2(n-j)+i}{i-m}\binom{n-j-m}{i-m}(i-m)!\cdot b_{n-j,m}\Bigg)
\\&=\binom{2n+k}{n}\cdot f_{n,k}-\sum_{j=1}^n \sum_{s=0}^k \sum_{m=0}^s
\frac{(-1)^{2m}\cdot(j+k-s)\cdot(n-j-m)!\cdot(k+2j-m-1)!}{2^{k-s}\cdot (j-k+s)!\cdot(k-s)!\cdot j!\cdot (s-m)!\cdot(n-j-s)!}\cdot b_{n-j,m},
\end{align*}
\end{small}
where the last equality is obtained from \texttt{Maple}.
\end{proof}

\end{document}